\documentclass[11pt]{article}

\usepackage{amsmath,amsthm,amscd,amsfonts,amssymb}
\usepackage{graphicx,slashbox}
\usepackage{subfigure}
\usepackage{a4wide}

\newtheorem{thm}{Theorem}[section]
\newtheorem{example}[thm]{Example}
\newtheorem{remark}[thm]{Remark}
\newtheorem{lemma}[thm]{Lemma}

\newtheorem{definition}[thm]{Definition}

\def\a{\alpha}
\def\t{\tau}
\def\LD{{_0D_t^\a}}
\def\LDa{{_aD_t^\a}}
\def\RD{{_tD_b^\a}}
\def\LDz{{_0D_t^{0.5}}}

\newenvironment{keywords}{\begin{center}
\begin{minipage}[c]{13.4cm} {\bf Keywords:}} {\end{minipage}
\end{center}}

\newenvironment{msc}{\begin{center}
\begin{minipage}[c]{13.4cm} {\bf MSC 2010:}} {\end{minipage}
\end{center}}

\begin{document}

\title{Numerical Approximations of Fractional Derivatives\\
with Applications\footnote{This is a preprint of a paper 
whose final and definite form will be published in: 
Asian Journal of Control. Submitted 13-Oct-2011; revised 11-Apr-2012; 
accepted 10-Aug-2012.} \thanks{Part of first author's Ph.D.,
which is carried out at the University of Aveiro under
the \emph{Doctoral Program in Mathematics and Applications} (PDMA)
of Universities of Aveiro and Minho.}}

\author{Shakoor Pooseh\\
\texttt{spooseh@ua.pt}
\and Ricardo Almeida\\
\texttt{ricardo.almeida@ua.pt}
\and Delfim F. M. Torres\\
\texttt{delfim@ua.pt}}

\date{Center for Research and Development in Mathematics and Applications\\
Department of Mathematics, University of Aveiro, 3810-193 Aveiro, Portugal}

\maketitle


\begin{abstract}
Two approximations, derived from continuous expansions
of Riemann--Liouville fractional derivatives into series involving
integer order derivatives, are studied. Using those series, one can formally transform
any problem that contains fractional derivatives into a classical problem
in which only derivatives of integer order are present. Corresponding
approximations provide useful numerical tools to compute fractional derivatives of functions.
Application of such approximations to fractional differential equations
and fractional problems of the calculus of variations are discussed.
Illustrative examples show the advantages and disadvantages of each approximation.
\end{abstract}

\begin{msc}
26A33, 33F05, 34A08, 49M99, 65D20.
\end{msc}



\section{Introduction}

Fractional calculus is the study of integrals and derivatives of arbitrary real or complex order.
Although the origin of fractional calculus goes back to the end of the seventeenth century,
the main contributions have been made during the last few decades \cite{Sun,Machado}.
Namely it has proven to be a useful tool when applied
to engineering and optimal control problems (see, \textrm{e.g.}, \cite{Efe,Li,Shen}).
There are several different definitions of fractional derivatives
in the literature, such as Gr\"{u}nwald--Letnikov, Caputo, etc.
Here we consider Riemann--Liouville fractional derivatives.

\begin{definition}[\textrm{cf.} \cite{Kilbas}]
Let $x(\cdot)$ be an absolutely continuous function in $[a,b]$ and $0<\a<1$. Then,
\begin{itemize}
\item the left Riemann--Liouville fractional derivative of order $\alpha$,
$\LDa$, is given by
\begin{equation}
\label{LeftD}
\LDa x(t)=\frac{1}{\Gamma(1-\alpha)}\frac{d}{dt}\int_a^t (t-\t)^{-\alpha}x(\t)d\t,
\quad t\in [a,b] ;
\end{equation}
\item the right Riemann--Liouville fractional derivative of order $\a$,
$\RD$, is given by
$$
\RD x(t)=\frac{1}{\Gamma(1-\a)}\left
(-\frac{d}{dt}\right )\int_t^b (\t-t)^{-\a}x(\t)d\t,
\quad t\in [a,b].
$$
\end{itemize}
\end{definition}

Due to the growing number of applications of fractional calculus in science and engineering
(see, \textrm{e.g.}, \cite{Das,Kai,Machado1}), numerical methods
are being developed to provide tools for solving such problems.
Using the Gr\"{u}nwald--Letnikov approach, it is convenient to approximate
the fractional differentiation operator, $D^\a$, by generalized finite differences.
In \cite{Pudlobny} some problems have been solved by this approximation.
In \cite{Kai1} a predictor-corrector method is presented that converts an initial value problem
into an equivalent Volterra integral equation,
while \cite{Agrawal2} shows the use of numerical methods
to solve such integral equations. A good survey on numerical methods
for fractional differential equations can be found in \cite{Ford}.

A new numerical scheme to solve fractional differential equations
has been recently introduced in \cite{Atan2} and \cite{Jelicic},
making an adaptation to cover fractional optimal control problems.
The scheme is based on an expansion formula
for the Riemann--Liouville fractional derivative.
Here we introduce a generalized version of that expansion and,
together with a different expansion formula that has been used
to approximate the fractional Euler--Lagrange equation in \cite{Atan1},
we perform an investigation of the advantages and disadvantages
of approximating fractional derivatives by these expansions.
The approximations transform fractional derivatives into finite sums
containing only derivatives of integer order.
We show the efficiency of such approximations to evaluate fractional derivatives
of a given function in closed form. Moreover, we discuss the possibility
of evaluating fractional derivatives of discrete tabular data.
The application to fractional differential equations and the calculus of variations
is also developed through some concrete examples. In each case we try to analyze problems
for which the analytic solution is available. This approach gives us the ability
of measuring the accuracy of each method. To this end, we need to measure
how close we get to exact solutions. We use the $2$-norm and define
the error function $E[x(\cdot),\tilde{x}(\cdot)]$ by
$$
E=\|x(\cdot)-\tilde{x}(\cdot)\|_2
=\left(\int_a^b [x(t)-\tilde{x}(t)]^2dt\right)^{\frac{1}{2}},
$$
where $x(\cdot)$ is defined on $[a,b]$.
The results of the paper give interesting numerical procedures
when applied to fractional problems of the calculus of variations.


\section{Expansion formulas to approximate fractional derivatives}
\label{secexpan}

In this section two approximations for the left Riemann--Liouville derivative are presented.
Both approximate the fractional derivatives by finite sums including
only derivatives of integer order and are based
on continuous expansions for the left Riemann--Liouville derivative.


\subsection{Approximation by a sum of integer order derivatives}

The right-hand side of \eqref{LeftD} is expandable in a power series
involving integer order derivatives \cite{Atan1,Samko}.
Let $(c,d)$, $-\infty<c<d<+\infty$, be an open interval in $\mathbb R$,
and $[a,b]\subset(c,d)$ be such that for each $t\in[a,b]$ the closed ball
$B_{b-a}(t)$, with center at $t$ and radius $b-a$, lies in $(c,d)$.
For any real analytic function $x(\cdot)$ in $(c,d)$
we can give the following expansion formula:
\begin{equation}
\label{expanIntInf}
\LDa x(t)=\sum_{n=0}^\infty \binom{\a}{n}\frac{(t-a)^{n-\a}}{\Gamma(n+1-\a)}x^{(n)}(t),
\quad \text{ where } \binom{\a}{n}=\frac{(-1)^{n-1}\a\Gamma(n-\a)}{\Gamma(1-\a)\Gamma(n+1)}.
\end{equation}
The condition $B_{b-a}(t)\subset(c,d)$ comes from the Taylor expansion of $x(t-\t)$ at $t$,
for $\t\in(a,t)$ and $t\in(a,b)$. The proof of this statement that can be found in \cite{Samko}
uses a similar expansion for fractional integrals. Here we outline a direct proof
due to our requirements in Section~\ref{ErrorSec}. Since $x(t)$ is analytic,
it can be expanded as a convergent power series, \textrm{i.e.},
$$
x(\t)=\sum_{n=0}^{\infty}\frac{(-1)^n x^{(n)}(t)}{n!}(t-\t)^n
$$
and then by \eqref{LeftD}
\begin{equation}
\label{expanIntErr}
\LDa x(t)=\frac{1}{\Gamma(1-\alpha)}\frac{d}{dt}
\int_a^t \left((t-\t)^{-\alpha}\sum_{n=0}^{\infty}
\frac{(-1)^n x^{(n)}(t)}{n!}(t-\t)^n\right)d\t,
\quad t\in (a,b).
\end{equation}
Termwise integration, followed by differentiation and simplification, leads to
\begin{equation*}
\LDa x(t)=\frac{x(t)}{\Gamma(1-\a)}(t-a)^{-\alpha}
+\frac{1}{\Gamma(1-\a)}\sum_{n=1}^\infty \left(\frac{(-1)^{n-1}}{(n-\a)(n-1)!}
+\frac{(-1)^n}{n!}\right)x^{(n)}(t)(t-a)^{n-\a}
\end{equation*}
and finally to expansion formula \eqref{expanIntInf}.
From the computational point of view, one can take only
a finite number of terms in \eqref{expanIntInf}
and use the approximation
\begin{eqnarray}
\label{expanInt}
\LDa x(t)\simeq\sum_{n=0}^N C(n,\a)(t-a)^{n-\a}x^{(n)}(t),
\quad \text{ where } C(n,\a)=\binom{\a}{n}\frac{1}{\Gamma(n+1-\a)}.
\end{eqnarray}


\subsection{Approximation using moments of a function}

The following lemma gives the departure point to another expansion.
For a proof see \cite{Kai}.

\begin{lemma}[Lemma 2.12 of \cite{Kai}]
Let $x(\cdot)\in AC[a,b]$ and $0 < \a < 1$. Then the left Riemann--Liouville
fractional derivative $\LDa x(\cdot)$ exists almost everywhere in $[a,b]$.
Moreover, $\LDa x(\cdot)\in L_p[a,b]$ for $1\leq p<\frac{1}{\a}$ and
\begin{equation}
\label{def}
\LDa x(t)=\frac{1}{\Gamma(1-\a)}\left[ \frac{x(a)}{(t-a)^\a}
+\int_a^t (t-\tau)^{-\a}\dot{x}(\tau)d\tau \right], \qquad t\in (a,b).
\end{equation}
\end{lemma}

Let $V_p(x(\cdot))$, $p\in \mathbb{N}$, denote the $(p-2)$th moment
of a function $x(\cdot)\in AC^2[a,b]$ (\textrm{cf.} \cite{Atan2}):
\begin{equation}
\label{defVp}
V_p(t):=V_p(x(t))=(1-p)\int_a^t (\tau-a)^{p-2}x(\tau) d\tau,
\quad p\in\mathbb{N}, \ t\geq a.
\end{equation}
Following \cite{Atan2}, it is easy to show that,
by successive integrating by parts,
\eqref{def} is reduced to
\begin{equation}
\label{exp2}
\LD x(t)=\frac{x(a)}{\Gamma(1-\a)}(t-a)^{-\a}
+\frac{\dot{x}(a)}{\Gamma(2-\a)}(t-a)^{1-\a}
+\frac{(t-a)^{1-\a}}{\Gamma(2-\a)}
\int_a^t \left(1-\frac{\t-a}{t-a}\right)^{1-\a}\ddot{x}(\t)d\t.
\end{equation}
Using the binomial theorem we conclude that
\begin{eqnarray}
\label{exp3}
\LDa x(t)&=&\frac{x(a)}{\Gamma(1-\a)}(t-a)^{-\a}
+\frac{(t-a)^{1-\a}}{\Gamma(2-\a)}\dot{x}(a)\nonumber\\
& &+\frac{(t-a)^{1-\a}}{\Gamma(2-\a)}\int_a^t
\left(\sum_{p=0}^{\infty}\frac{\Gamma(p-1+\a)}{\Gamma(\a-1)p!}
\left(\frac{\t-a}{t-a}\right)^p\right)\ddot{x}(\t)d\t, \quad t>a.
\end{eqnarray}
Further integration by parts and simplification in \eqref{exp3} gives
\begin{equation}
\label{expanMomInf}
\LDa x(t)=A(\a)(t-a)^{-\a}x(t)+B(\a)(t-a)^{1-\a}\dot{x}(t)
-\sum_{p=2}^{\infty}C(\a,p)(t-a)^{1-p-\a}V_p(t),
\end{equation}
where $V_p(t)$ is defined by \eqref{defVp} and
\begin{eqnarray*}
A(\a)&=&\frac{1}{\Gamma(1-\a)}\left[1+\sum_{p
=2}^{\infty}\frac{\Gamma(p-1+\a)}{\Gamma(\a)(p-1)!}\right],\\
B(\a)&=&\frac{1}{\Gamma(2-\a)}\left[1+\sum_{p
=1}^{\infty}\frac{\Gamma(p-1+\a)}{\Gamma(\a-1)p!}\right],\\
C(\a,p)&=&\frac{1}{\Gamma(2-\a)\Gamma(\a-1)}\frac{\Gamma(p-1+\a)}{(p-1)!}.
\end{eqnarray*}
The moments $V_p(t)$, $p=2,3,\ldots$, are regarded as the solutions
to the following system of differential equations:
\begin{equation}
\label{sysVp}
\left\{
\begin{array}{l}
\dot{V}_p(t)=(1-p)(t-a)^{p-2}x(t)\\
V_p(a)=0, \qquad p=2,3,\ldots
\end{array}
\right.
\end{equation}
For numerical purposes, only a finite number of terms in the series
\eqref{expanMomInf} are used. We approximate the fractional derivative as
\begin{equation}
\label{expanMom}
\LDa x(t)\simeq A(\a,N)(t-a)^{-\a}x(t)+B(\a,N)(t-a)^{1-\a}\dot{x}(t)
-\sum_{p=2}^N C(\a,p)(t-a)^{1-p-\a}V_p(t),
\end{equation}
where $A(\a,N)$ and $B(\a,N)$ are given by
\begin{eqnarray}
A(\a,N)&=&\frac{1}{\Gamma(1-\a)}\left[1+\sum_{p
=2}^N\frac{\Gamma(p-1+\a)}{\Gamma(\a)(p-1)!}\right],\label{A}\\
B(\a,N)&=&\frac{1}{\Gamma(2-\a)}\left[1+\sum_{p
=1}^N\frac{\Gamma(p-1+\a)}{\Gamma(\a-1)p!}\right]\label{B}.
\end{eqnarray}

\begin{remark}
Our approximation \eqref{expanMom} is different from the
one presented in \cite{Atan2}: since the infinite series
$\sum_{p=1}^{\infty}\frac{\Gamma(p-1+\a)}{\Gamma(\a-1)p!}$
tends to $-1$, $B(\a)=0$ and thus
\begin{equation}
\label{expanAtan}
\LD x(t)\simeq A(\a,N)t^{-\a}x(t)-\sum_{p=2}^N C(\a,p)t^{1-p-\a}V_p(t).
\end{equation}
However, regarding the fact that we use a finite sum, in practice one has
\begin{equation*}
1+\sum_{p=1}^N\frac{\Gamma(p-1+\a)}{\Gamma(\a-1)p!}\neq 0.
\end{equation*}
Therefore, and similarly to \cite{Djordjevic,MyID:225},
we keep here the approximation in the form \eqref{expanMom}.
The value of $B(\a,N)$ for some values of $N$ and for different choices of $\a$
is given in Table~\ref{tab}.
\begin{table}[!ht]
\center
\begin{tabular}{|c|c c c c c c c|}
\hline
$N$         &    4   &    7   &   15   &   30   &   70   &  120   &  170   \\
\hline
$B(0.1,N)$  & 0.0310 & 0.0188 & 0.0095 & 0.0051 & 0.0024 & 0.0015 & 0.0011 \\
$B(0.3,N)$  & 0.1357 & 0.0928 & 0.0549 & 0.0339 & 0.0188 & 0.0129 & 0.0101 \\
$B(0.5,N)$  & 0.3085 & 0.2364 & 0.1630 & 0.1157 & 0.0760 & 0.0581 & 0.0488 \\
$B(0.7,N)$  & 0.5519 & 0.4717 & 0.3783 & 0.3083 & 0.2396 & 0.2040 & 0.1838 \\
$B(0.9,N)$  & 0.8470 & 0.8046 & 0.7481 & 0.6990 & 0.6428 & 0.6092 & 0.5884 \\
$B(0.99,N)$ & 0.9849 & 0.9799 & 0.9728 & 0.9662 & 0.9582 & 0.9531 & 0.9498 \\
\hline
\end{tabular}
\caption{$B(\a,N)$ for different values of  $\a$ and $N$.}
\label{tab}
\end{table}
It shows that even for a large $N$, when $\a$ tends to one,
$B(\a,N)$ cannot be ignored. In Figure~\ref{BaN}
we plot $B(\a,N)$ as a function of $N$ for different values of $\a$.
\begin{figure}[!ht]
  \begin{center}
    \subfigure{\includegraphics[scale=0.5]{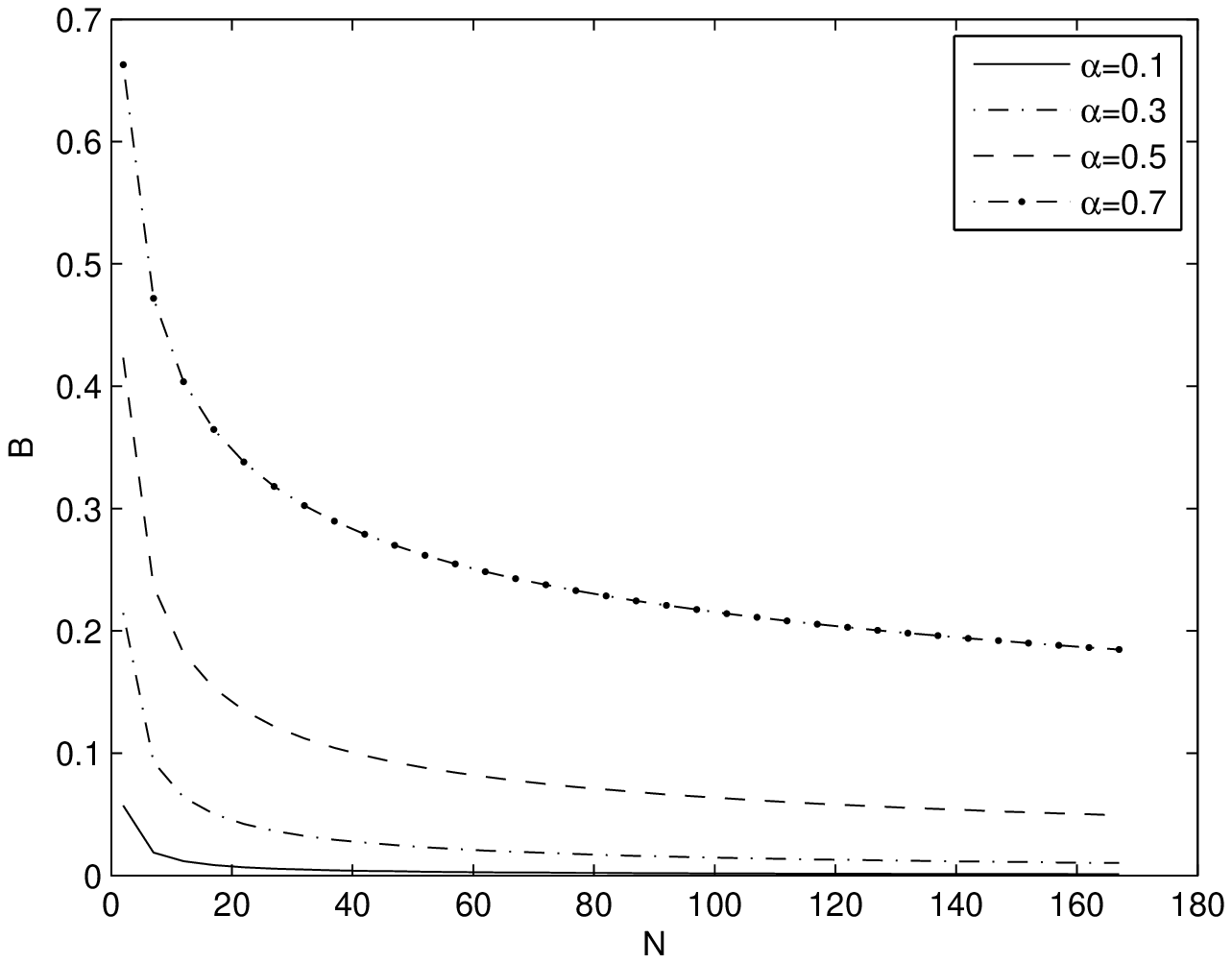}}
    \subfigure{\includegraphics[scale=0.5]{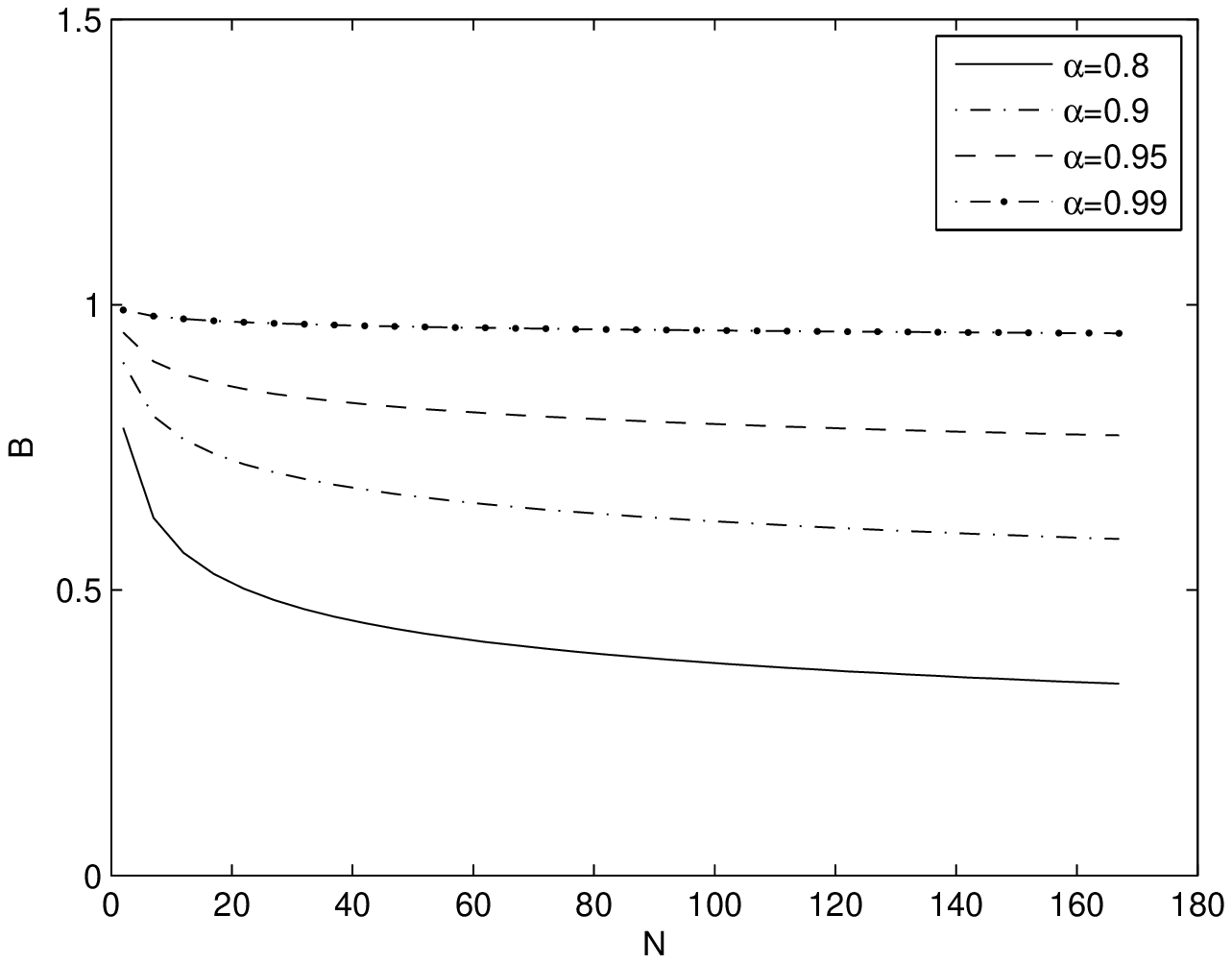}}
  \end{center}
  \caption{$B(\a,N)$ for different values of  $\a$ and $N$.}
  \label{BaN}
\end{figure}
In Section~\ref{sec3} we compare
both approximations with some examples.
We also refer to \cite{Atan0,MyID:221},
where such type of expansion formulas are studied.
\end{remark}

A similar argument gives the expansion formula for $\RD$,
the right Riemann--Liouville fractional derivative.
We propose the following approximation:
\begin{equation*}
\RD x(t)\simeq A(\a,N)(b-t)^{-\a}x(t)-B(\a,N)(b-t)^{1-\a}\dot{x}(t)
-\sum_{p=2}^NC(\a,p)(b-t)^{1-p-\a}W_p(t),
\end{equation*}
where $W_p(t)=(1-p)\int_t^b (b-\tau)^{p-2}x(\tau)d\tau$.
Here $A(\a,N)$ and $B(\a,N)$ are the same as \eqref{A}
and \eqref{B}, respectively.

Formula \eqref{expanMomInf} consists of two parts:
an infinite series and two terms including
the first derivative and the function itself.
It can be generalized to contain derivatives of higher-order.

\begin{thm}
\label{thm:oft}
Fix $n \in \mathbb{N}$ and let $x(\cdot)\in C^n[a,b]$. Then,
\begin{multline}
\label{Gen}
\LDa x(t)=\frac{1}{\Gamma(1-\a)}(t-a)^{-\a}x(t)+\sum_{i=1}^{n-1}
A(\a,i)(t-a)^{i-\a}x^{(i)}(t)\\
+\sum_{p=n}^{\infty}\left[
\frac{-\Gamma(p-n+1+\a)}{\Gamma(-\a)\Gamma(1+\a)(p-n+1)!}(t-a)^{-\a}x(t)
+ B(\a,p)(t-a)^{n-1-p-\a}V_p(t)\right],
\end{multline}
where
\begin{eqnarray*}
A(\a,i)&=&\frac{1}{\Gamma(i+1-\a)}\left[1+\sum_{p
=n-i}^{\infty}\frac{\Gamma(p-n+1+\a)}{\Gamma(\a-i)(p-n+i+1)!}\right],
\quad i = 1,\ldots,n-1,\\
B(\a,p)&=&\frac{\Gamma(p-n+1+\a)}{\Gamma(-\a)\Gamma(1+\a)(p-n+1)!},\\
V_p(t) &=&(p-n+1)\int_a^t (\t-a)^{p-n}x(\t)d\t.
\end{eqnarray*}
\end{thm}

\begin{proof}
Successive integrating by parts in \eqref{def} gives
\begin{eqnarray*}
\LDa x(t)&=&\frac{x(a)}{\Gamma(1-\a)}(t-a)^{-\a}
+\frac{\dot{x}(a)}{\Gamma(2-\a)}(t-a)^{1-\a}
+\cdots+\frac{x^{(n-1)}(a)}{\Gamma(n-\a)}(t-a)^{n-1-\a}\\
& &+\frac{1}{\Gamma(1-\a)}\int_a^t (t-\tau)^{n-1-\a}x^{(n)}(\tau)d\tau.
\end{eqnarray*}
Using the binomial theorem, we expand the integral term as
$$
\int_a^t (t-\tau)^{n-1-\a}x^{(n)}(\tau)d\tau=(t-a)^{n-1-\a}\sum_{p=0}^{\infty}
\frac{\Gamma(p-n+1+\a)}{\Gamma(1-n+\a)p!(t-a)^p}\int_a^t (\tau-a)^px^{(n)}(\tau)d\tau.
$$
Splitting the sum into $p=0$ and $p=1\ldots\infty$,
and integrating by parts the last integral, we get
\begin{eqnarray*}
\LD x(t)&=&\frac{(t-a)^{-\a}}{\Gamma(1-\a)}x(a)+\cdots
+\frac{(t-a)^{n-2-\a}}{\Gamma(n-1-\a)}x^{(n-2)}(a)\\
& &+\frac{(t-a)^{n-1-\a}}{\Gamma(n-\a)}x^{(n-2)}(t)\left[1
+ \sum_{p=1}^\infty\frac{\Gamma(p-n+1+\a)}{\Gamma(-n+1+\a)p!}\right]\\
& &+ \frac{(t-a)^{n-1-\a}}{\Gamma(n-1-\a)} \sum_{p=1}^\infty\frac{\Gamma(p
-n+1+\a)}{\Gamma(-n+2+\a)(p-1)!(t-a)^p}\int_a^t(\tau-a)^{p-1}x^{(n-1)}(\tau)d\tau.\\
\end{eqnarray*}
The rest of the proof follows a similar routine, \textrm{i.e.},
by splitting the sum into two parts, the first term and the rest,
and integrating by parts the last integral until
$x(\cdot)$ appears in the integrand.
\end{proof}

\begin{remark}
The series that appear in $A(\a,i)$ are convergent
for all $ i \in \{1,\ldots,n-1\}$. Fix an $i$ and observe that
$$
\sum_{p=n-i}^{\infty}\frac{\Gamma(p-n+1+\a)}{\Gamma(\a-i)(p-n+i+1)!}
=\sum_{p=1}^{\infty}\frac{\Gamma(p+\a-i)}{\Gamma(\a-i)p!}
={_1F_0}(\a-i,1)-1.
$$
Since $i>\a$, ${_1F_0}(\a-i,1)$ converges by Theorem~2.1.1 of \cite{Andrews}.
In practice we only use finite sums and for $A(\a,i)$ we can easily compute
the truncation error. Although this is a partial error, it gives a good
intuition of why this approximation works well.
Using the fact that ${_1F_0}(a,1)=0$ if $a<0$
(\textrm{cf.} Eq. (2.1.6) in \cite{Andrews}), we have
\begin{equation*}
\begin{split}
\frac{1}{\Gamma(i+1-\a)}&\sum_{p=N+1}^{\infty}\frac{\Gamma(p
-n+1+\a)}{\Gamma(\a-i)(p-n+i+1)!}\\
&=\frac{1}{\Gamma(i+1-\a)}\left({_1F_0}(\a-i,1)
- \sum_{p=0}^{N-n+i+1}\frac{\Gamma(p+\a-i)}{\Gamma(\a-i)p!}\right)\\
&=\frac{-1}{\Gamma(i+1-\a)}\sum_{p=0}^{N-n+i+1}\frac{\Gamma(p+\a-i)}{\Gamma(\a-i)p!}.
\end{split}
\end{equation*}
In Table~\ref{truncError} we give some values for this error,
with $\a=0.5$ and different values for $i$ and $N-n$.
\begin{table}[!ht]
\center
\begin{tabular}{|c|c|c|c|c|c|}
\hline
 \backslashbox{$i$}{$N-n$}& 0 & 5 & 10 & 15 & 20  \\
 \hline
1 &-0.4231&-0.2364&-0.1819&-0.1533& -0.1350\\
 \hline
2 &0.04702&0.009849&0.004663&0.002838&0.001956\\
 \hline
3 &-0.007052&-0.0006566&-0.0001999&-0.00008963&-0.00004890\\
 \hline
4 &0.001007&0.00004690&0.000009517&0.000003201&0.000001397\\
 \hline
\end{tabular}
\caption{Truncation errors of $A(\a,i,N)$ for $\a=0.5$.}
\label{truncError}
\end{table}
\end{remark}

\begin{remark}
Using Euler's reflection formula, one can define $B(\a,p)$
of Theorem~\ref{thm:oft} as
$$
B(\a,p)=\frac{-sin(\pi\a)\Gamma(p-n+1+\a)}{\pi (p-n+1)!}.
$$
\end{remark}

For numerical purposes, only finite sums are taken
to approximate fractional derivatives.
Therefore, for a fixed $n \in \mathbb{N}$ and $N\geq n$, one has
\begin{equation}\label{ApproxDerGeneralcase}
\LDa x(t)\approx \sum_{i=0}^{n-1} A(\a,i,N)(t-a)^{i-\a}x^{(i)}(t)
+\sum_{p=n}^{N} B(\a,p)(t-a)^{n-1-p-\a}V_p(t),
\end{equation}
where
\begin{eqnarray*}
A(\a,i,N)&=&\frac{1}{\Gamma(i+1-\a)}\left[1+\sum_{p=2}^{N}\frac{\Gamma(p
-n+1+\a)}{\Gamma(\a-i)(p-n+i+1)!}\right], \quad i = 0,\ldots,n-1,\\
B(\a,p)&=&\frac{\Gamma(p-n+1+\a)}{\Gamma(-\a)\Gamma(1+\a)(p-n+1)!},\\
V_p(t) &=&(p-n+1)\int_a^t (\t-a)^{p-n}x(\t)d\t.
\end{eqnarray*}

Similarly, we can deduce an expansion formula for the right fractional derivative.

\begin{thm}
Fix $n\in\mathbb N$ and $x(\cdot)\in C^n[a,b]$. Then,
\begin{multline*}
\RD x(t)=\frac{1}{\Gamma(1-\a)}(b-t)^{-\a} x(t)
+\sum_{i=1}^{n-1}A(\alpha,i)(b-t)^{i-\a} x^{(i)}(t)\\
+\sum_{p=n}^\infty\left[ \frac{-\Gamma(p-n+1+\a)}{\Gamma(-\a)\Gamma(1
+\a)(p-n+1)!}(b-t)^{-\a} x(t)+B(\a,p)(b-t)^{n-1-\a-p}W_p(t)\right],
\end{multline*}
where
\begin{eqnarray*}
A(\a,i)&=&\frac{(-1)^i}{\Gamma(i+1-\a)}\left[1+\sum_{p
=n-i}^\infty\frac{\Gamma(p-n+1+\a)}{\Gamma(-i+\a)(p-n+1+i)!}\right],
\quad i = 1,\ldots,n-1,\\
B(\a,p)&=&\frac{(-1)^n\Gamma(p-n+1+\a)}{\Gamma(-\a)\Gamma(1+\a)(p-n+1)!},\\
W_p(t)&=&(p-n+1)\int_t^b (b-\tau)^{p-n}x(\tau)d\tau.
\end{eqnarray*}
\end{thm}

\begin{proof}
Analogous to the proof of Theorem~\ref{thm:oft}.
\end{proof}


\subsection{Error estimation}
\label{ErrorSec}

This section is devoted to the study of the error caused by choosing a finite number
of terms in the expansions. For the expansion \eqref{expanIntInf},
we separate the error term in \eqref{expanIntErr} and rewrite it as
\begin{eqnarray}
\label{IntErr}
\LDa x(t)&=&\frac{1}{\Gamma(1-\alpha)}\frac{d}{dt}
\int_a^t \left((t-\t)^{-\alpha}\sum_{n=0}^{N}\frac{(-1)^n
x^{(n)}(t)}{n!}(t-\t)^n\right)d\t\nonumber\\
& &+\frac{1}{\Gamma(1-\alpha)}\frac{d}{dt}
\int_a^t \left((t-\t)^{-\alpha}\sum_{n=N+1}^{\infty}\frac{(-1)^n
x^{(n)}(t)}{n!}(t-\t)^n\right)d\t.
\end{eqnarray}
The first term in \eqref{IntErr} gives \eqref{expanInt} directly
and the second term is the error caused by truncation. The next step
is to give a local upper bound for this error, $E_{tr}(t)$.
The series
$$
\sum_{n=N+1}^{\infty}\frac{(-1)^n x^{(n)}(t)}{n!}(t-\t)^n,
\quad \t\in (a,t), \quad t\in (a,b),
$$
is the remainder of the Taylor expansion of $x(\t)$
and thus bounded by $\left|\frac{M}{(N+1)!}(t-\t)^{N+1}\right|$
in which $M=\displaystyle\max_{\t \in [a,t]} \left|x^{(N+1)}(\t)\right|$. Then,
$$
E_{tr}(t)\leq \left|\frac{M}{\Gamma(1-\alpha)(N+1)!}\frac{d}{dt}
\int_a^t (t-\t)^{N+1-\alpha}d\t\right|
=\frac{M}{\Gamma(1-\alpha)(N+1)!}(t-a)^{N+1-\alpha}.
$$

For approximation \eqref{expanMom}, we observe that
the integrand in \eqref{exp2} can be expanded,
by the binomial theorem, as
\begin{eqnarray}
\label{expanError}
\left(1-\frac{\t-a}{t-a}\right)^{1-\a}&=&\sum_{p=0}^{\infty}
\frac{\Gamma(p-1+\a)}{\Gamma(\a-1)p!}\left(\frac{\t-a}{t-a}\right)^p\nonumber\\
&=&\sum_{p=0}^{N}\frac{\Gamma(p-1+\a)}{\Gamma(\a
-1)p!}\left(\frac{\t-a}{t-a}\right)^p+R_N(\t),
\end{eqnarray}
where
$$
R_N(\t)=\sum_{p=N+1}^{\infty}\frac{\Gamma(p
-1+\a)}{\Gamma(\a-1)p!}\left(\frac{\t-a}{t-a}\right)^p.
$$
Substituting \eqref{expanError} into \eqref{exp2}, we get
\begin{eqnarray*}
\LD x(t)&=&\frac{x(a)}{\Gamma(1-\a)}(t-a)^{-\a}
+\frac{\dot{x}(a)}{\Gamma(2-\a)}(t-a)^{1-\a}\\
&& +\frac{(t-a)^{1-\a}}{\Gamma(2-\a)}\int_a^t
\left(\sum_{p=0}^{N}\frac{\Gamma(p-1+\a)}{\Gamma(\a-1)p!}\left(
\frac{\t-a}{t-a}\right)^p+R_N(\t)\right)\ddot{x}(\t)d\t\\
&=&\frac{x(a)}{\Gamma(1-\a)}(t-a)^{-\a}+\frac{\dot{x}(a)}{\Gamma(2-\a)}(t-a)^{1-\a}\\
&& +\frac{(t-a)^{1-\a}}{\Gamma(2-\a)}\int_a^t
\left(\sum_{p=0}^{N}\frac{\Gamma(p-1+\a)}{\Gamma(\a-1)p!}
\left(\frac{\t-a}{t-a}\right)^p\right)\ddot{x}(\t)d\t\\
&& +\frac{(t-a)^{1-\a}}{\Gamma(2-\a)}\int_a^t R_N(\t)\ddot{x}(\t)d\t.
\end{eqnarray*}
At this point, we apply the techniques of \cite{Atan2}
to the first three terms with finite sums.
Then, we receive \eqref{expanMom} with
an extra term of truncation error:
\begin{equation*}
E_{tr}(t)=\frac{(t-a)^{1-\a}}{\Gamma(2-\a)}
\int_a^t R_N(\t)\ddot{x}(\t)d\t.
\end{equation*}
Since $0\leq\frac{\t-a}{t-a}\leq 1$ for $\t\in [a,t]$, one has
\begin{eqnarray*}
|R_N(\t)|&\leq & \sum_{p=N+1}^{\infty}\left|
\frac{\Gamma(p-1+\a)}{\Gamma(\a-1)p!}\right|
=\sum_{p=N+1}^{\infty}\left|\binom{1-\a}{p}\right|
\leq\sum_{p=N+1}^{\infty}\frac{\mathrm{e}^{(1-\a)^2+1-\a}}{p^{2-\a}} \\
&\leq&\int_{p=N}^{\infty}\frac{\mathrm{e}^{(1-\a)^2+1-\a}}{p^{2-\a}}dp
=\frac{\mathrm{e}^{(1-\a)^2+1-\a}}{(1-\a)N^{1-\a}}.
\end{eqnarray*}
Finally, assuming $L_2=\displaystyle\max_{\t \in [a,t]}\left|\ddot{x}(\t)\right|$,
we conclude that
\begin{equation*}
|E_{tr}(t)|\leq
L_2 \frac{\mathrm{e}^{(1-\a)^2+1-\a}}{\Gamma(2-\a)(1-\a)N^{1-\a}} (t-a)^{2-\a}.
\end{equation*}

In the general case, the error is given by the following result.
\begin{thm}
If we approximate the left Riemann--Liouville fractional derivative by the finite sum
\eqref{ApproxDerGeneralcase}, then the error $E_{tr}(\cdot)$ is bounded by
\begin{equation}
\label{eq:formula:rv1}
|E_{tr}(t)|\leq L_n
\frac{\mathrm{e}^{(n-1-\a)^2+n-1-\a}}{\Gamma(n-\a)(n-1-\a)N^{n-1-\a}}(t-a)^{n-\a},
\end{equation}
where
$$
L_n=\displaystyle\max_{\t \in [a,t]}\left|x^{(n)}(\t)\right|.
$$
\end{thm}
From \eqref{eq:formula:rv1} we see that if
the test function grows very fast or the point $t$ is far from $a$,
then the value of $N$ should also increase in order to have a good approximation.
Clearly, if we increase the value of $n$,
then we need also to increase the value of $N$ to control the error.


\section{Numerical evaluation of fractional derivatives}
\label{sec3}

In \cite{Pudlobny} a numerical method to evaluate fractional derivatives
is given based on the Gr\"{u}nwald--Letnikov definition of fractional derivatives.
It uses the fact that for a large class of functions, the Riemann--Liouville
and the Gr\"{u}nwald--Letnikov definitions are equivalent.
We claim that the approximations discussed so far provide a good tool
to compute numerically the fractional derivatives of given functions.
For functions whose higher-order derivatives are easily available,
we can freely choose between approximations \eqref{expanInt} or \eqref{expanMom}.
But in the case that difficulties arise in computing higher-order derivatives,
we choose the approximation \eqref{expanMom} that needs only the values
of the first derivative and function itself. Even if the first derivative
is not easily computable, we can use the approximation given by \eqref{expanAtan}
with large values for $N$ and $\a$ not so close to one. As an example, we compute
$\LD x(t)$, with $\a=\frac{1}{2}$, for $x(t)=t^4$ and $x(t)=e^{2t}$.
The exact formulas of the derivatives are derived from
\begin{equation*}
\LDz(t^n)=\frac{\Gamma(n+1)}{\Gamma(n+1-0.5)}t^{n-0.5}
\quad \text{and} \quad \LDz(e^{\lambda t})
=t^{-0.5}E_{1,1-0.5}(\lambda t),
\end{equation*}
where $E_{\a,\beta}$ is the two parameter
Mittag--Leffler function \cite{Pudlobny}.
Figure~\ref{EvalInt} shows the results using approximation \eqref{expanInt}.
As we can see, the third approximations are reasonably accurate for both cases.
Indeed, for $x(t)=t^4$, the approximation with $N=4$ coincides with
the exact solution because the derivatives of order five and more vanish.
\begin{figure}[!ht]
  \begin{center}
    \subfigure[$\LDz(t^4)$]{\label{figaNumEvalInt}\includegraphics[scale=0.5]{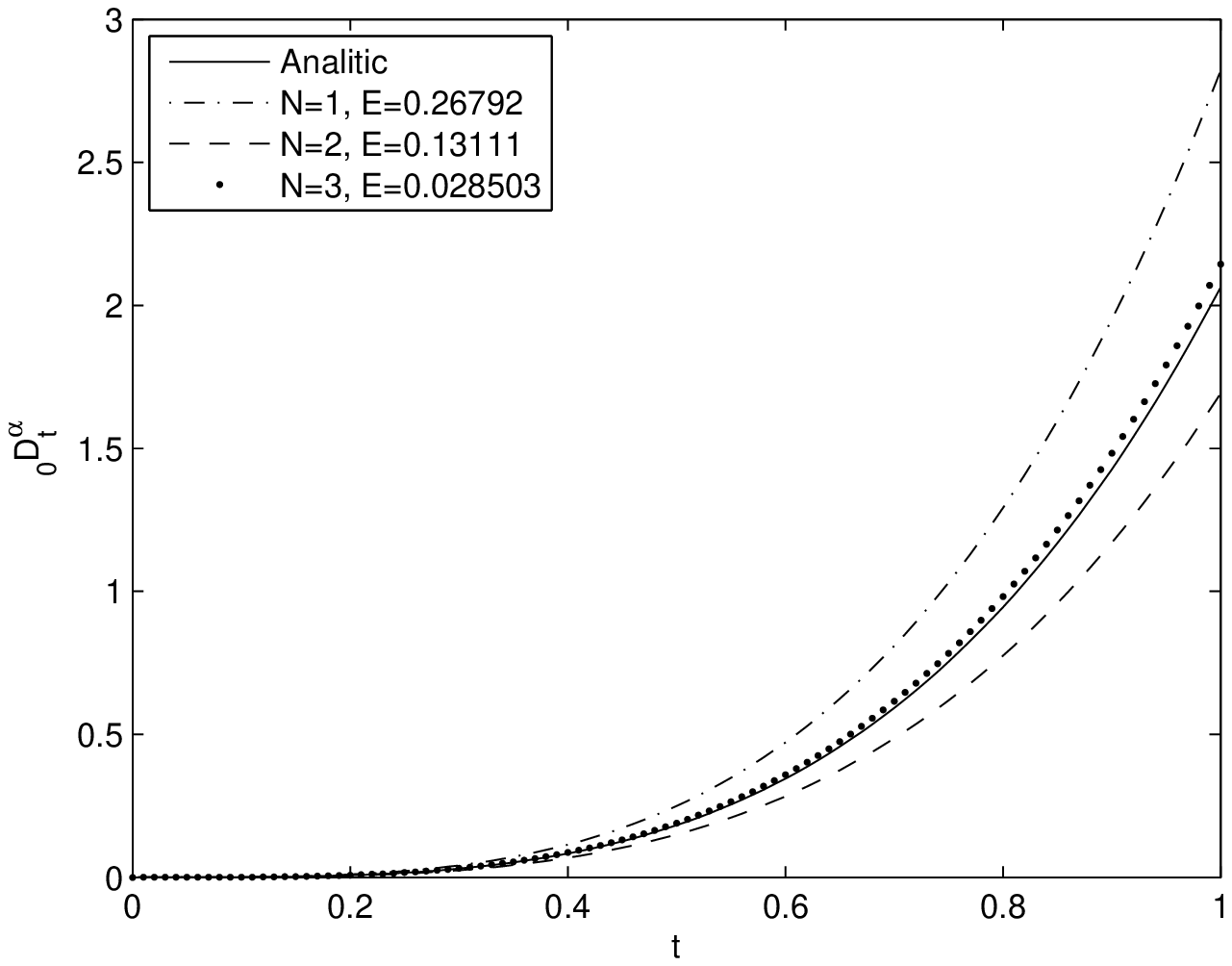}}
    \subfigure[$\LDz(e^{2t})$]{\label{figbNumEvalInt}\includegraphics[scale=0.5]{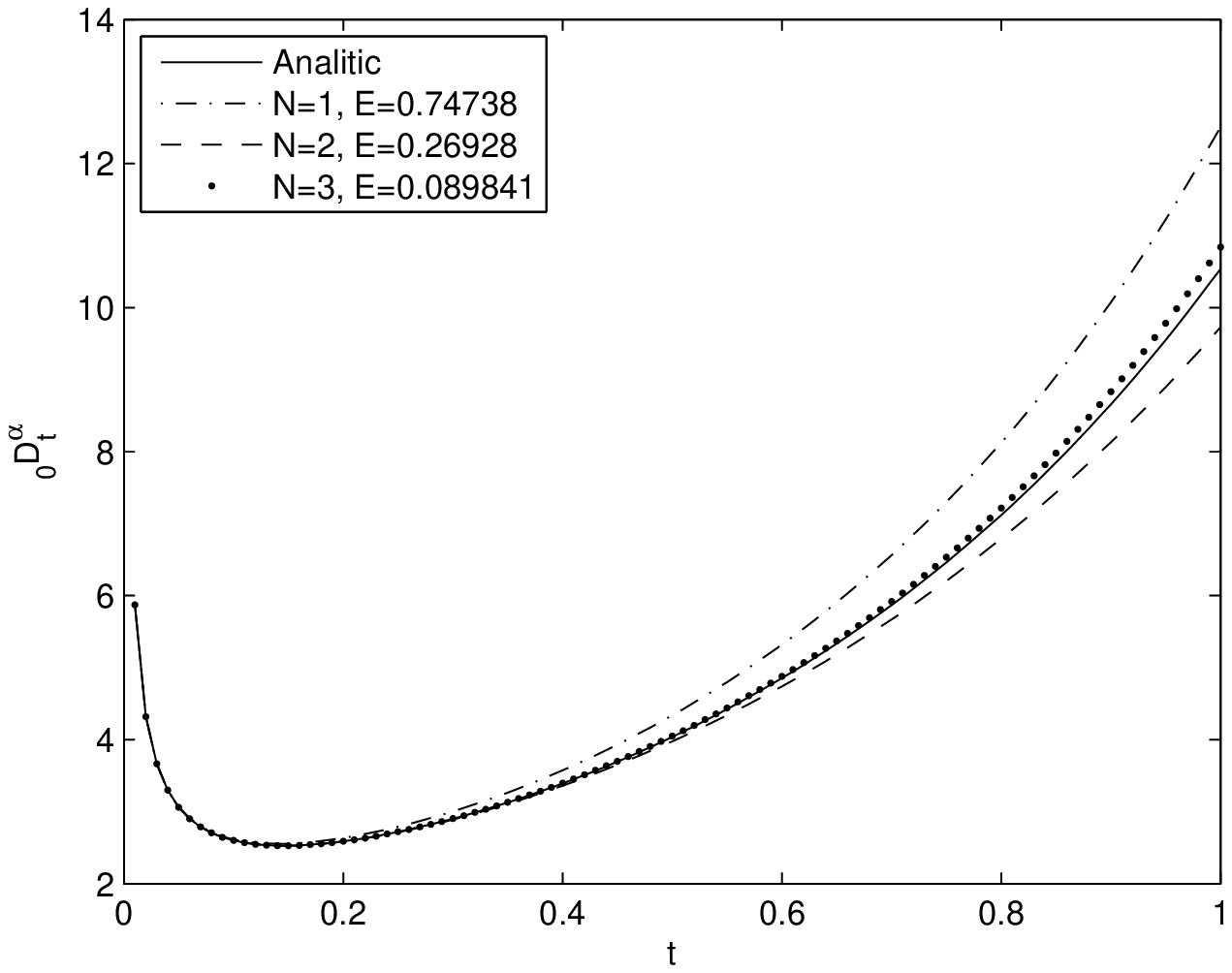}}
  \end{center}
  \caption{Analytic (solid line) versus numerical approximation \eqref{expanInt}.}
  \label{EvalInt}
\end{figure}
The same computations are carried out using approximation \eqref{expanMom}.
In this case, given a function $x(\cdot)$, we can compute $V_p$ by definition
or integrate the system \eqref{sysVp} analytically or by any numerical integrator.
As it is clear from Figure~\ref{EvalMom},
one can get better results by using larger values of $N$.
\begin{figure}[!ht]
  \begin{center}
    \subfigure[$\LDz(t^4)$]{\includegraphics[scale=0.5]{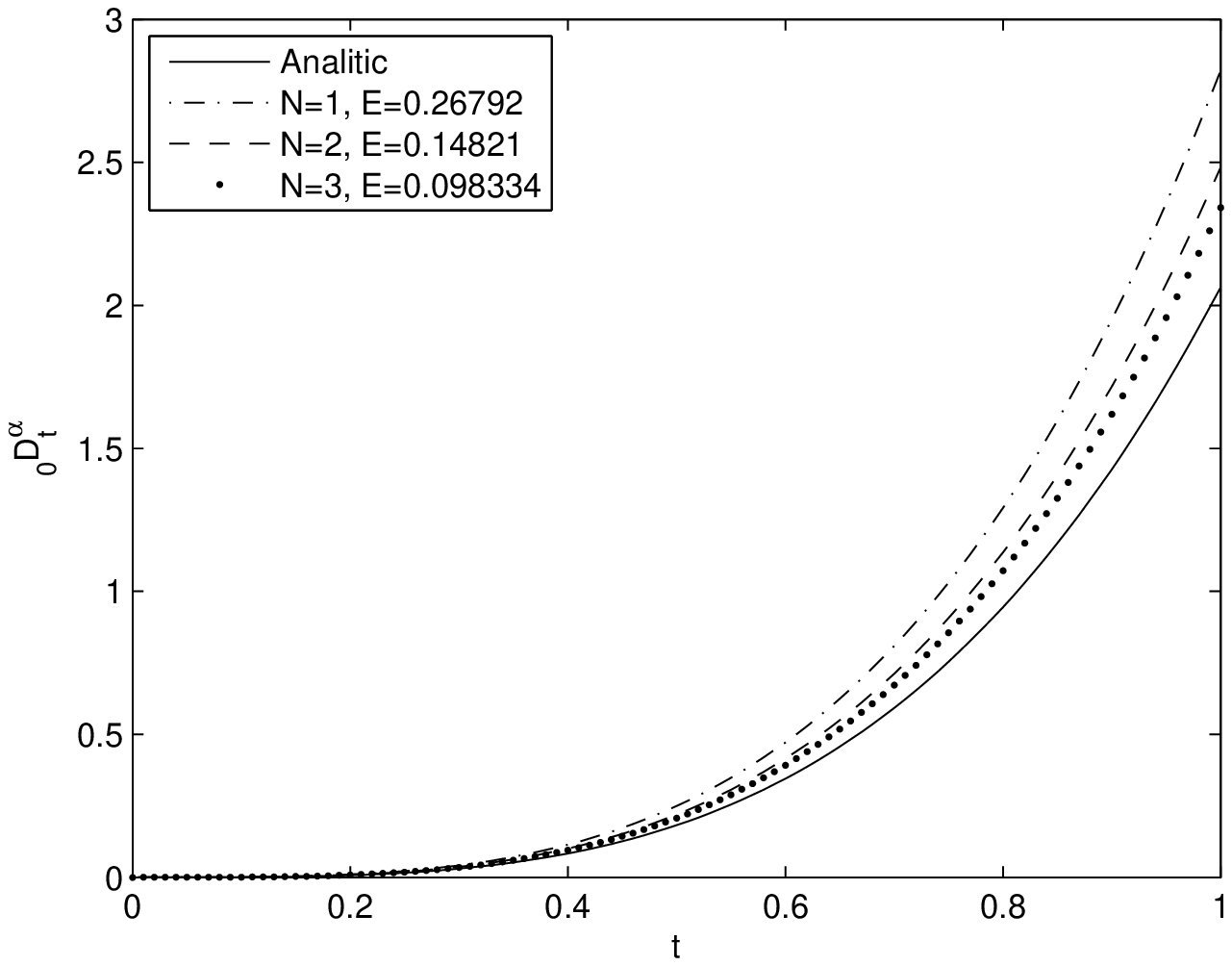}}
    \subfigure[$\LDz(e^{2t})$]{\includegraphics[scale=0.5]{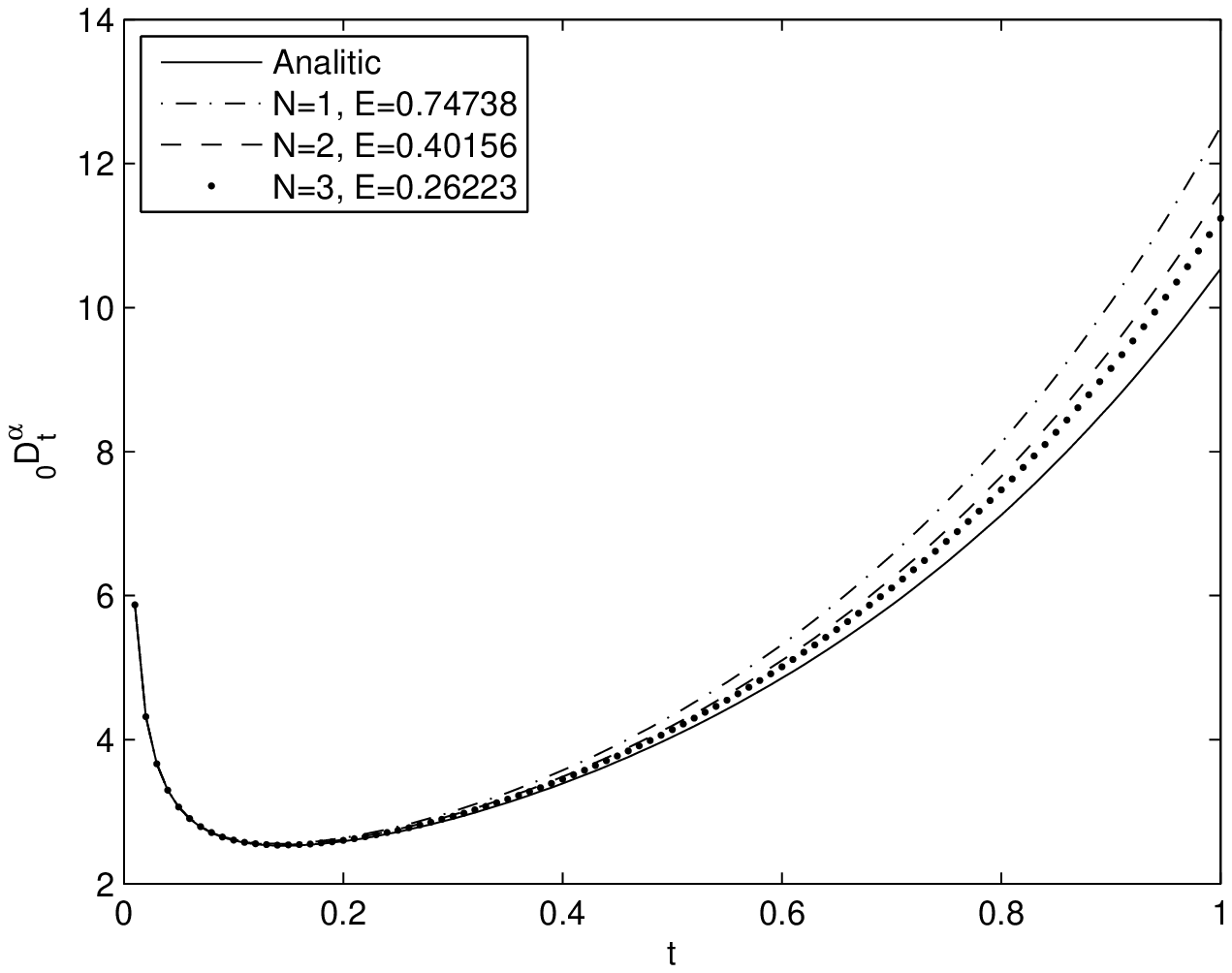}}
  \end{center}
  \caption{Analytic  (solid line) versus numerical approximation \eqref{expanMom}.}
  \label{EvalMom}
\end{figure}
Comparing Figures~\ref{EvalInt} and \ref{EvalMom}, we find out that the approximation
\eqref{expanInt} shows a faster convergence. Observe that both functions are analytic
and it is easy to compute higher-order derivatives. The approximation \eqref{expanInt}
fails for non-analytic functions as stated in \cite{Atan2}.

\begin{remark}
\label{EndPoints}
A closer look to \eqref{expanInt} and \eqref{expanMom} reveals
that in both cases the approximations are not computable
at $a$ and $b$ for the left and right fractional derivatives, respectively.
At these points we assume that it is possible to extend them continuously
to the closed interval $[a,b]$.
\end{remark}

In what follows, we show that by omitting the first derivative
from the expansion, as done in \cite{Atan2}, one may loose
a considerable accuracy in computation. Once again, we compute
the fractional derivatives of $x(t)=t^4$ and $x(t)=e^{2t}$,
but this time we use the approximation given by \eqref{expanAtan}.
Figure~\ref{compareEval} summarizes the results. Our expansion
gives a more realistic approximation using quite small $N$,
$3$ in this case.
\begin{figure}[!ht]
\begin{center}
\subfigure[$\LDz(t^4)$]{\label{fig2a}\includegraphics[scale=0.5]{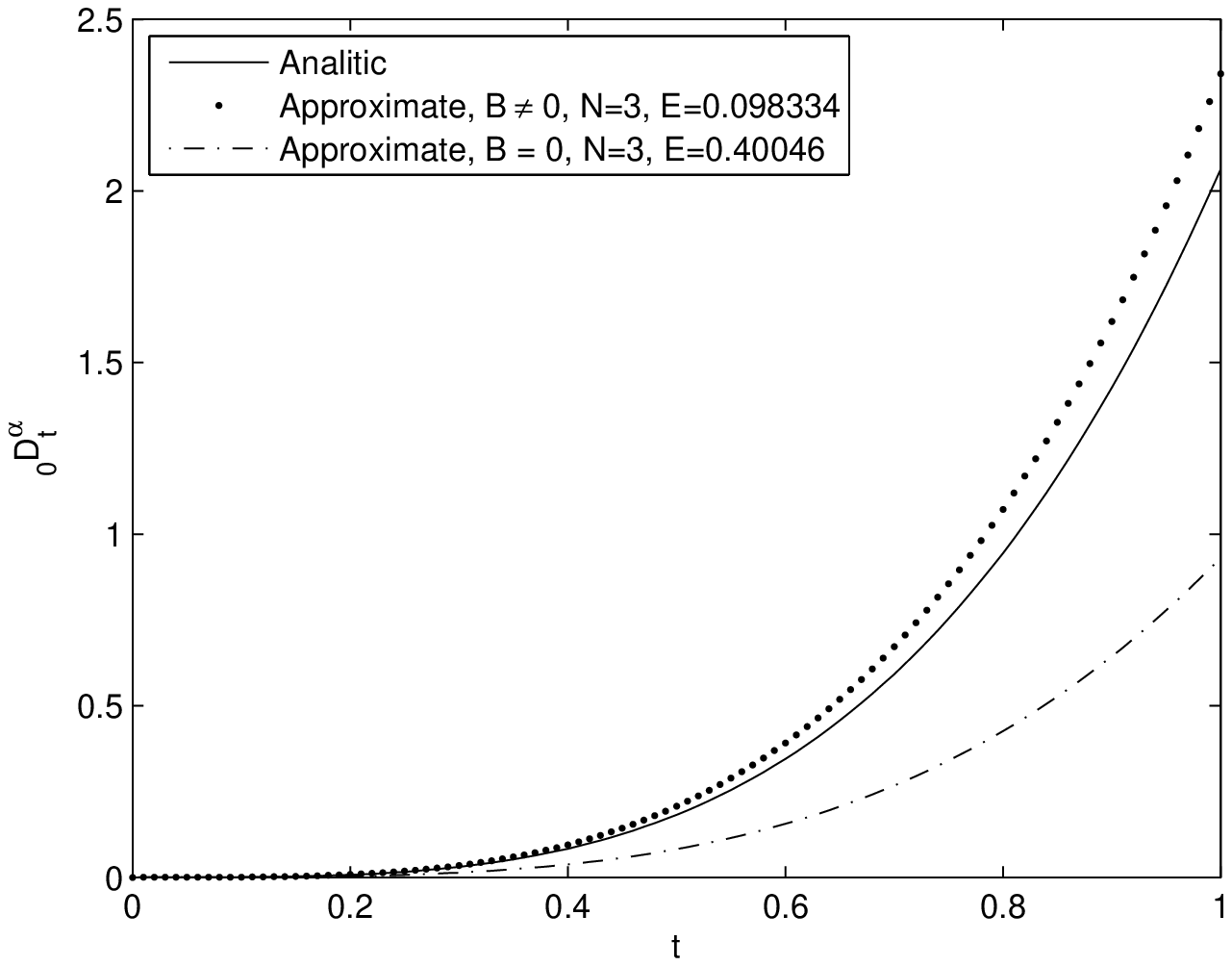}}
\subfigure[$\LDz(e^{2t})$]{\label{fig2d}\includegraphics[scale=0.5]{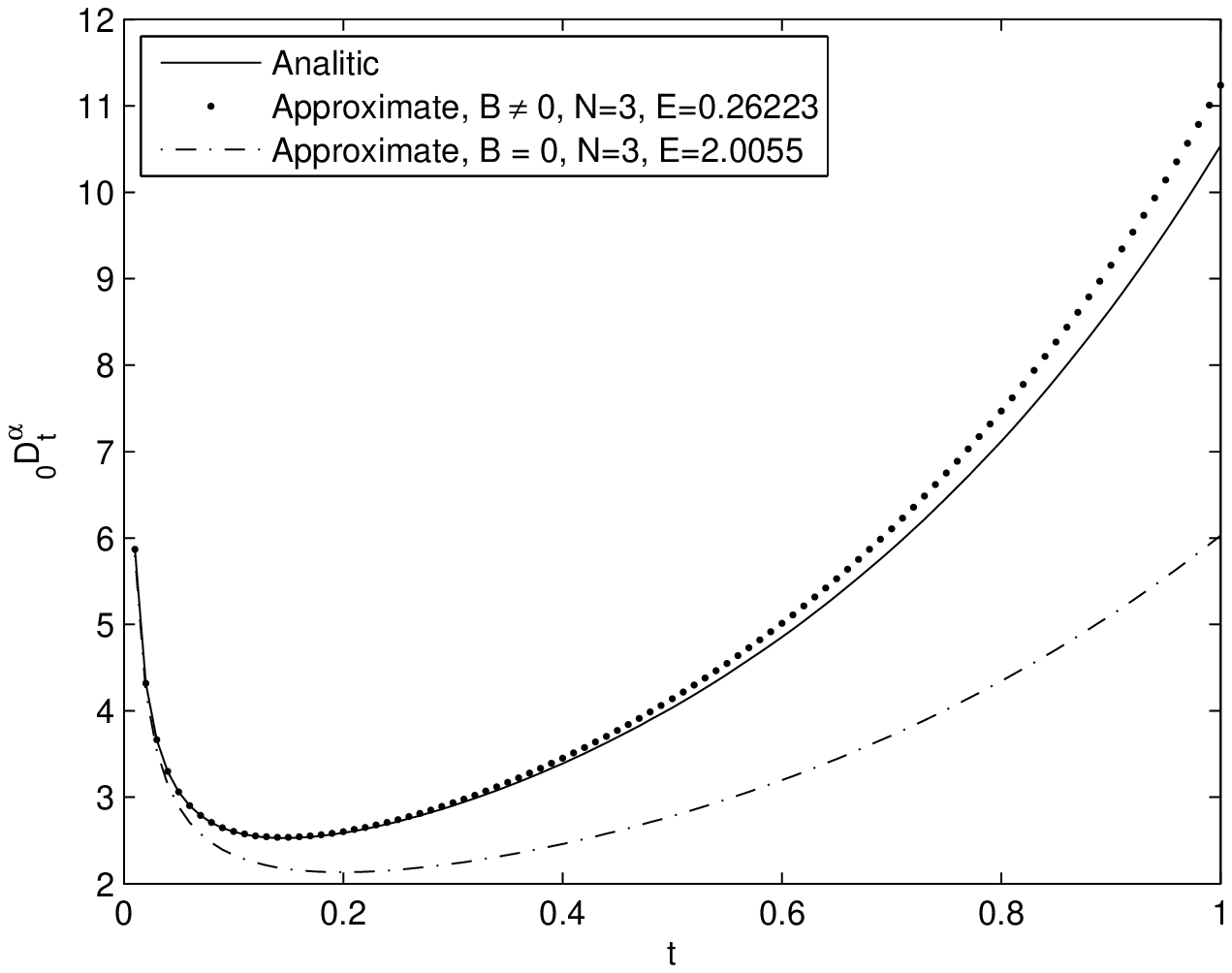}}
\end{center}
\caption{Comparison of approximation \eqref{expanMom} proposed here
and approximation \eqref{expanAtan} of \cite{Atan2}}
\label{compareEval}
\end{figure}
To show how the appearance of higher-order derivatives in generalization \eqref{Gen}
gives better results, we evaluate fractional derivatives of $x(t)=t^4$ and $x(t)=e^{2t}$
for different values of $n$. We consider $n=1,2,3$, $N=6$ for $x(t)=t^4$
(Figure~\ref{ExpEt}) and $N=4$ for $x(t)=e^{2t}$ (Figure~\ref{ExpExn}).
\begin{figure}[!ht]
\begin{center}
\subfigure[$\LDz(t^4)$]{\label{ExpEt}\includegraphics[scale=0.5]{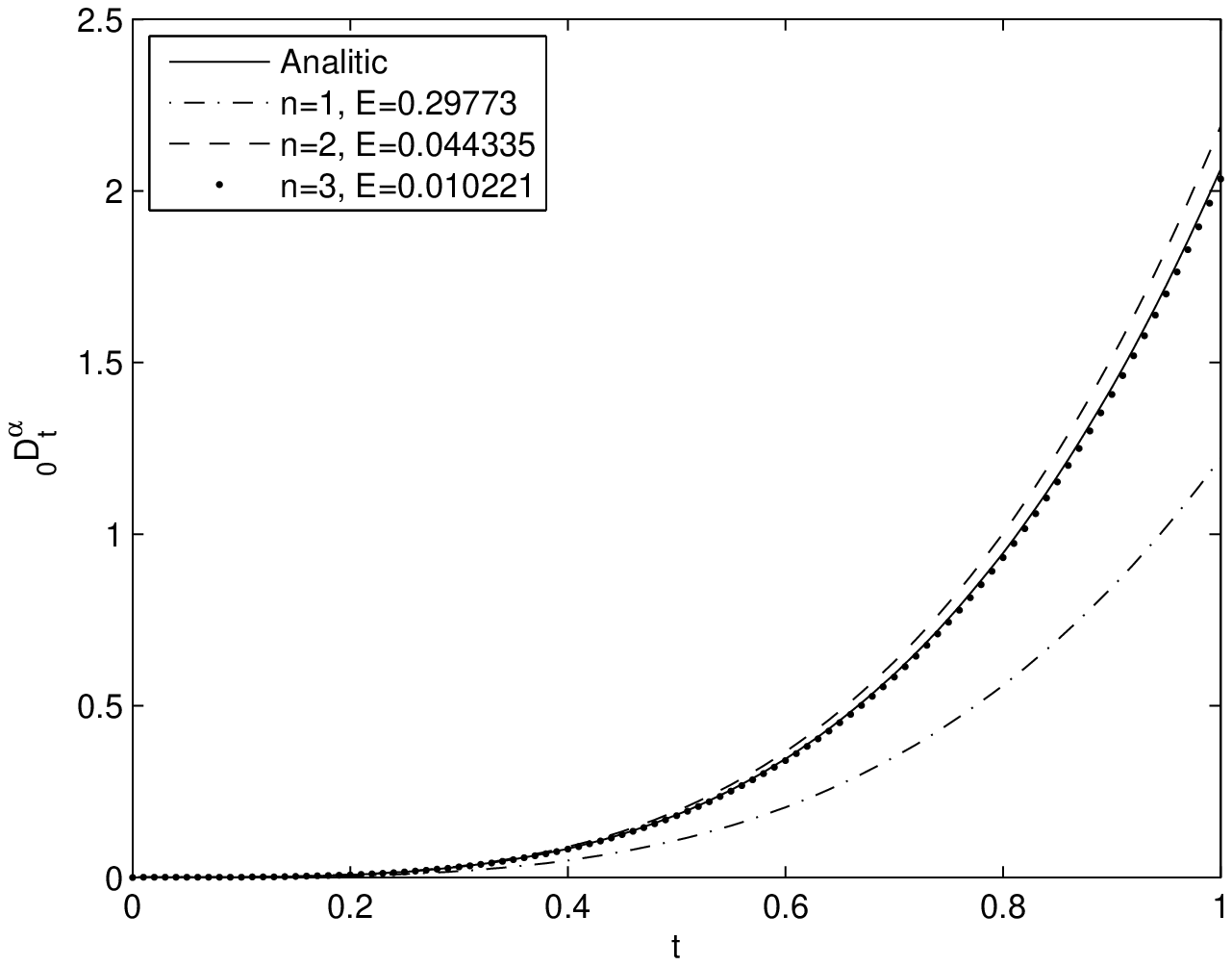}}
\subfigure[$\LDz(e^{2t})$]{\label{ExpExn}\includegraphics[scale=0.5]{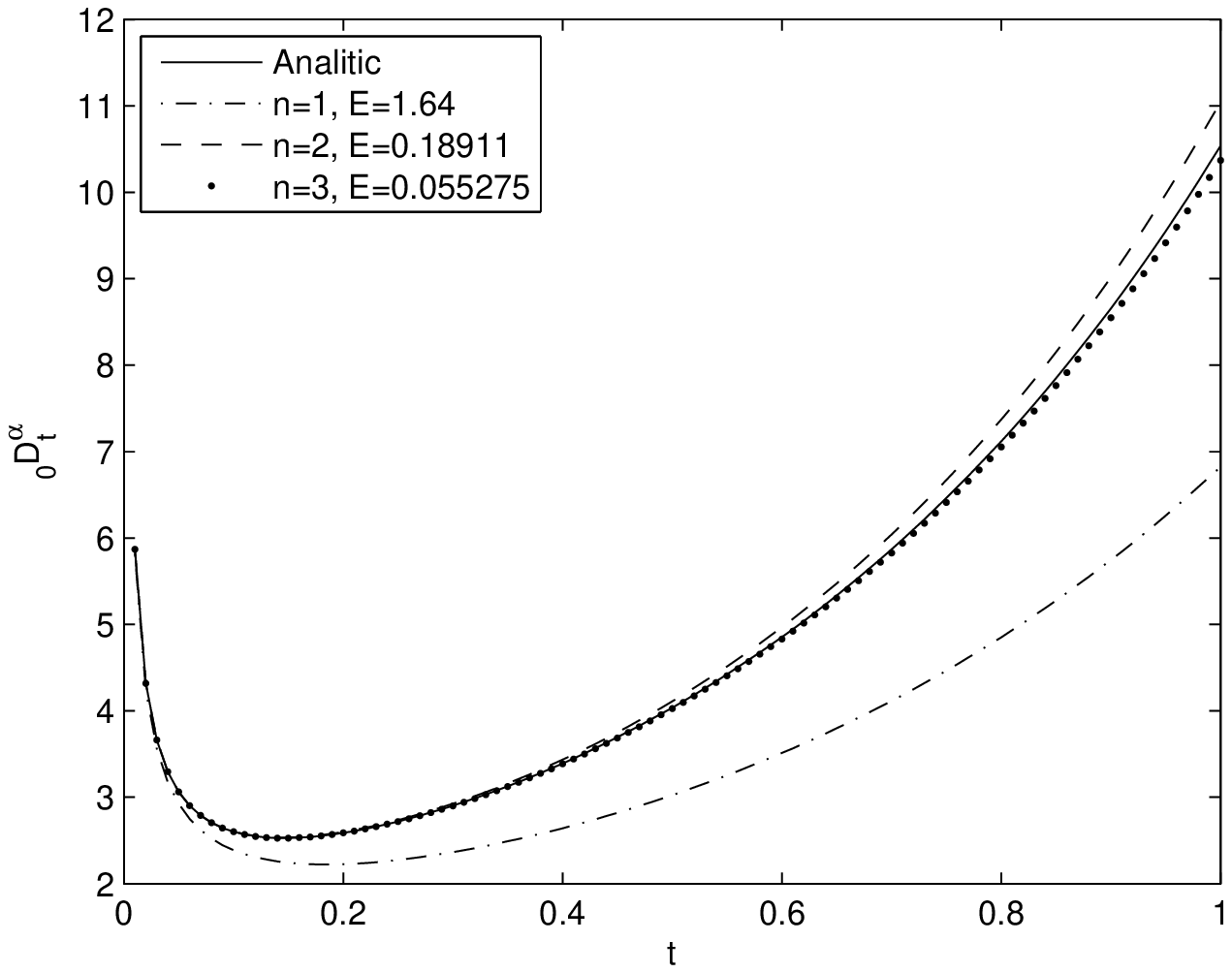}}
\end{center}
\caption{Analytic (solid line) versus numerical approximation \eqref{Gen}.}
\label{IntExp}
\end{figure}


\subsection{Fractional derivatives of tabular data}

In many applications (see Section~\ref{FCOV}),
the function itself is not accessible in a closed form, but
as a tabular data for discrete values of the independent variable.
Thus, we cannot use the definition to compute the fractional derivative directly.
Our approximation \eqref{expanMom}, that uses the function and its first derivative
to evaluate the fractional derivative, seems to be a good candidate in those cases.
Suppose that we know the values of $x(t_i)$ on $n+1$ distinct points in a given interval $[a,b]$,
\textrm{i.e.}, for $t_i$, $i=0,1,\ldots,n$, with $t_0=a$ and $t_n=b$.
According to formula \eqref{expanMom}, the value of the fractional derivative
of $x(\cdot)$ at each point $t_i$ is given approximately by
$$
\LDa x(t_i)\simeq A(\a,N)(t_i-a)^{-\a}x(t_i)
+B(\a,N)(t_i-a)^{1-\a}\dot{x}(t_i)
-\sum_{p=2}^NC(p,\a)(t_i-a)^{1-p-\a}V_p(t_i).
$$
The values of $x(t_i)$, $i=0,1,\ldots,n$, are given. A good approximation for
$\dot{x}(t_i)$ can be obtained using the forward,
centered, or backward difference approximation
of the first-order derivative \cite{Stoer}. For $V_p(t_i)$ one can either
use the definition and compute the integral numerically,
\textrm{i.e.}, $V_p(t_i)=\int_a^{t_i} (1-p)(\tau-a)^{p-2}x(\tau)d\tau$,
or it is possible to solve \eqref{sysVp} as an initial value problem.
All required computations are straightforward and only need
to be implemented with the desired accuracy. The only thing to take
care is the way of choosing a good order, $N$, in the formula \eqref{expanMom}.
Because no value of $N$, guaranteeing the error
to be smaller than a certain preassigned number, is known a priori,
we start with some prescribed value for $N$ and increase it step by step.
In each step we compare, using an appropriate norm,
the result with the one of previous step.
For instance, one can use the Euclidean norm
$\|(\LDa)^{new}-(\LDa)^{old} \|_2$ and terminate the procedure
when it's value is smaller than a predefined $\epsilon$.
For illustrative purposes, we compute the fractional derivatives
of order $\a=0.5$ for tabular data extracted from $x(t)=t^4$ and $x(t)=e^{2t}$.
The results are given in Figure~\ref{tabular}.
\begin{figure}[!ht]
  \begin{center}
    \subfigure[$\LDz(t^4)$]{\label{fig5a}\includegraphics[scale=0.5]{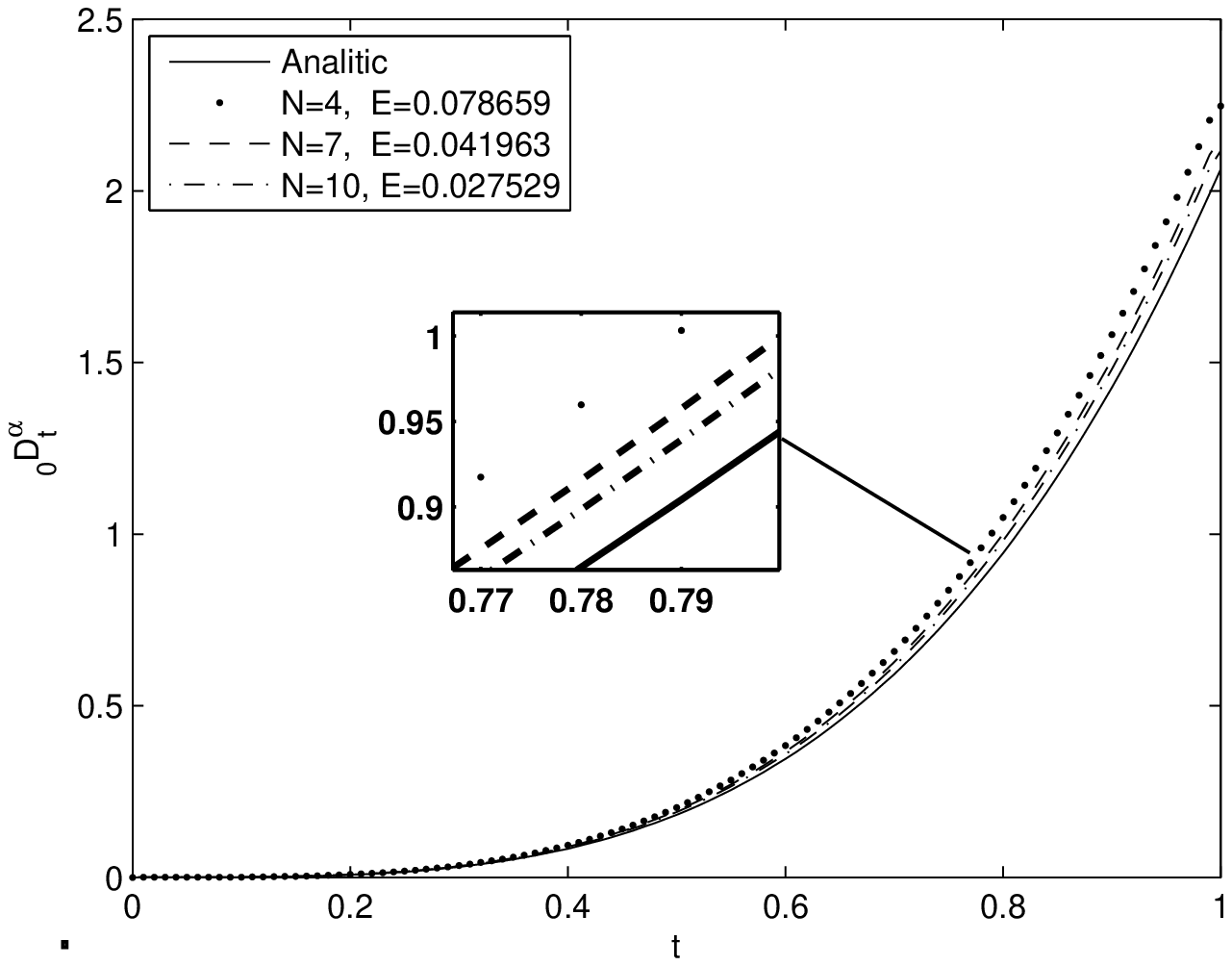}}
    \subfigure[$\LDz(e^{2t})$]{\label{fig5d}\includegraphics[scale=0.5]{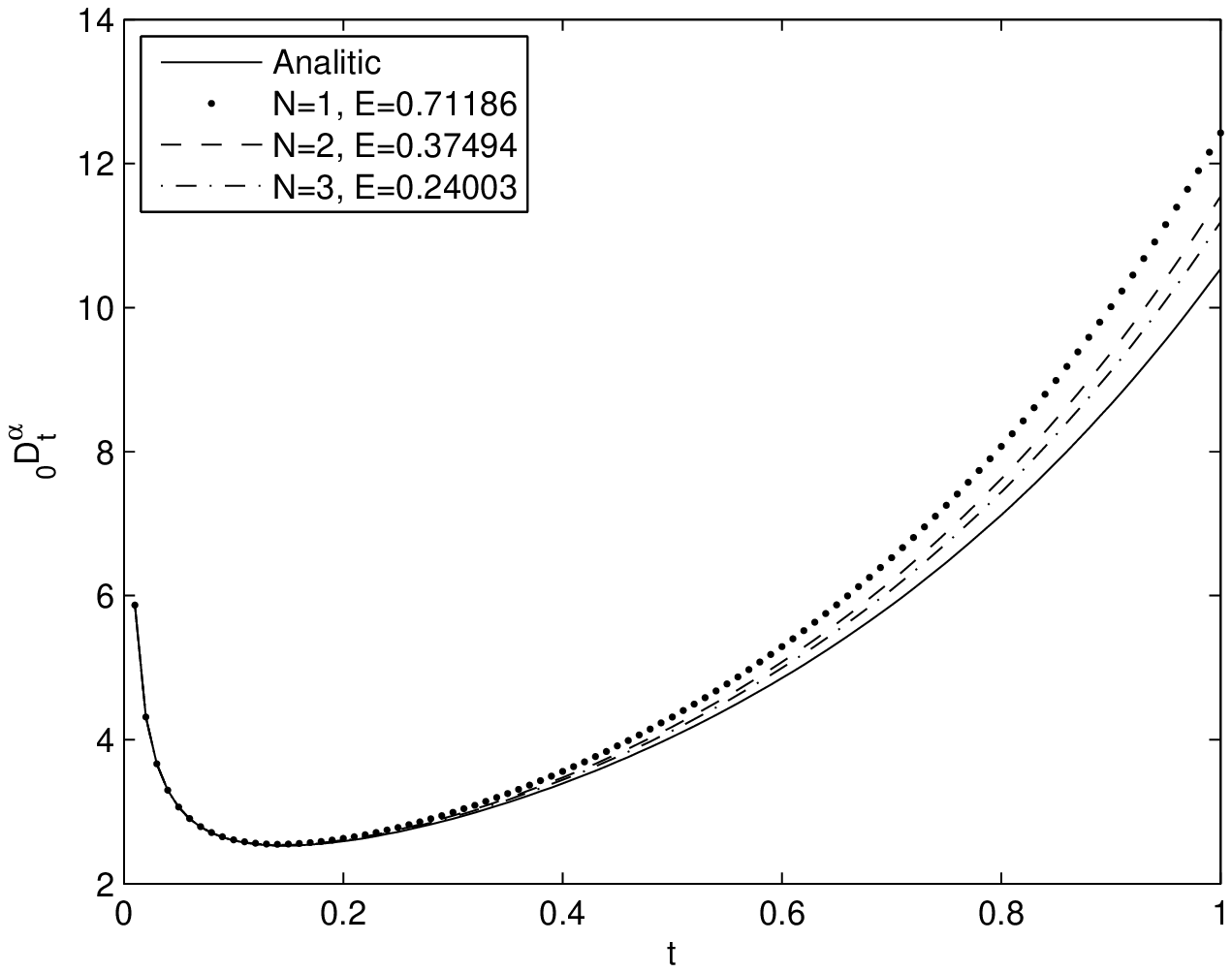}}
  \end{center}
  \caption{Fractional derivatives of tabular data}
  \label{tabular}
\end{figure}


\section{Numerical solution to fractional differential equations}

The classical theory of ordinary differential equations is a well developed field
with many tools available for numerical purposes. Using the approximations
\eqref{expanInt} and \eqref{expanMom}, one can transform a fractional
ordinary differential equation into a classical ODE.

We should mention here that, using \eqref{expanInt}, derivatives
of higher-order appear in the resulting ODE, while we only have
a limited number of initial or boundary conditions available.
In this case the value of $N$, the order of approximation,
should be equal to the number of given conditions.
If we choose a larger $N$, we will encounter
lack of initial or boundary conditions. This problem is not present
in the case in which we use the approximation \eqref{expanMom},
because the initial values for the auxiliary variables $V_p$, $p=2,3,\ldots$,
are known and we don't need any extra information.

Consider, as an example, the following initial value problem:
\begin{equation}
\label{ex1}
\left\{
\begin{array}{l}
\LDz x(t)+x(t)=t^2+\frac{2}{\Gamma(2.5)}t^{\frac{3}{2}},\\
x(0)=0.
\end{array}
\right.
\end{equation}
We know that $\LDz (t^2)=\frac{2}{\Gamma(2.5)}t^{\frac{3}{2}}$.
Therefore, the analytic solution for system \eqref{ex1} is $x(t)=t^2$.
Because only one initial condition is available, we can only expand
the fractional derivative up to the first derivative
in \eqref{expanInt}. One has
\begin{equation}
\label{ExampOdeInt}
\left\{
\begin{array}{l}
1.5642~t^{-0.5}x(t)+0.5642~t^{0.5}\dot{x}(t)=t^2+1.5045~t^{1.5},\\
x(0)=0.
\end{array}
\right.
\end{equation}
This is a classical initial value problem
and can be easily treated numerically.
The solution is drawn in Figure~\ref{odeInt}.
As expected, the result is not satisfactory.
Let us now use the approximation given by \eqref{expanMom}.
The system in \eqref{ex1} becomes
\begin{equation}
\label{ex2}
\left\{
\begin{array}{l}
A(N)t^{-0.5}x(t)+B(N)t^{0.5}\dot{x}(t)
-\sum_{p=2}^NC(p)t^{0.5-p}V_p+x(t)=t^2+\frac{2}{\Gamma(2.5)}t^{1.5},\\
\dot{V}_p(t)=(1-p)(t-a)^{p-2}x(t), \quad p=2,3,\ldots,N,\\
x(0)=0,\\
V_p(0)=0, \quad p=2,3,\ldots,N.
\end{array}
\right.
\end{equation}
We solve this initial value problem for $N=7$.
The Matlab \textsf{ode45} built-in function is used
to integrate system \eqref{ex2}. The solution is given
in Figure~\ref{odeVp} and shows a better approximation
when compared with \eqref{ExampOdeInt}.
\begin{figure}[!ht]
\begin{center}
\subfigure[Exact versus Approximation
\eqref{expanInt}]{\label{odeInt}\includegraphics[scale=0.5]{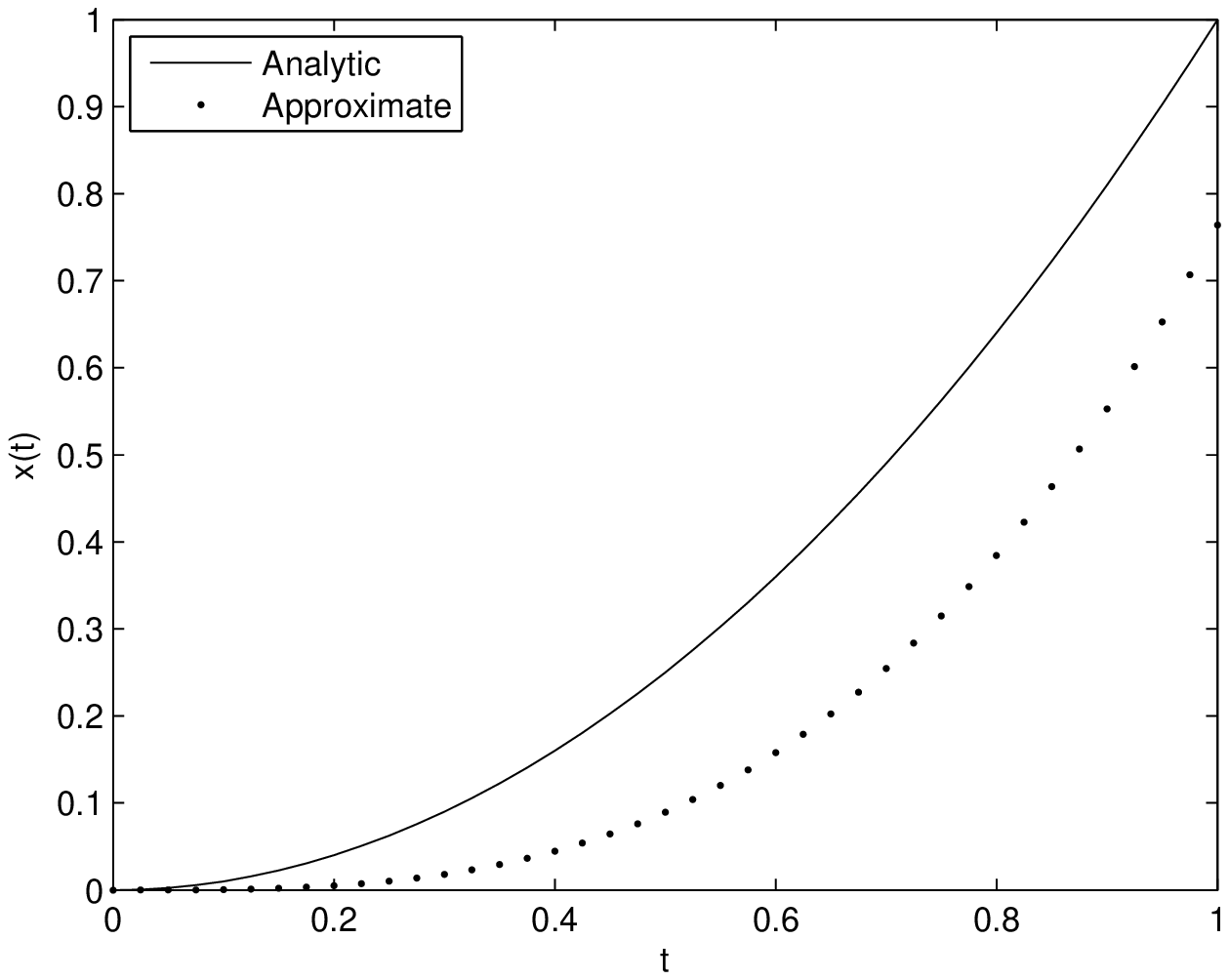}}
\subfigure[Exact versus Approximation \eqref{expanMom}]{\label{odeVp}
\includegraphics[scale=0.5]{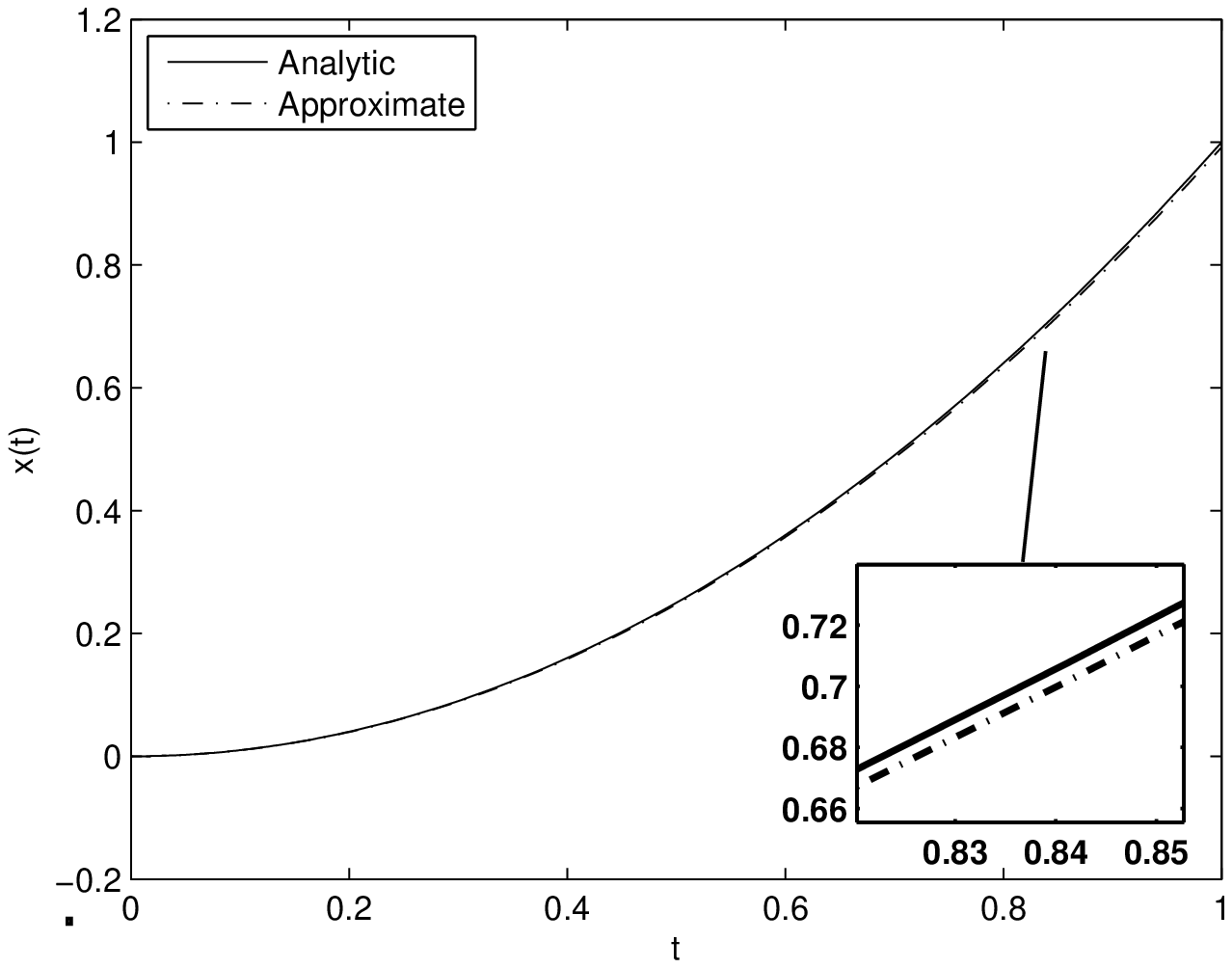}}\\
\end{center}
\caption{Two approximations applied to fractional differential equation \eqref{ex1}}
\end{figure}

\begin{remark}
To show the difference caused by the appearance of the first derivative
in formula \eqref{expanMom}, we solve the initial value problem \eqref{ex1}
with $B(\a,N)=0$. Since the original fractional differential equation
does not depend on integer order derivatives of function $x(\cdot)$,
\textrm{i.e.}, it has the form
$$
\LDa x(t)+f(x,t)=0,
$$
by \eqref{expanAtan} the dependence to derivatives of  $x(\cdot)$ vanishes.
In this case one needs to apply the operator $_aD_t^{1-\a}$
to the above equation and obtain
$$
\dot{x}(t)+_aD_t^{1-\a}[f(x,t)]=0.
$$
Nevertheless, we can use \eqref{expanMom} directly without any trouble.
Figure~\ref{cmpAtanMe} shows that at least for a moderate accurate method,
like the Matlab routine \textsf{ode45}, taking $B(\a,N)\neq 0$ into account
gives a better approximation.
\begin{figure}[!ht]
  \begin{center}
    \includegraphics[scale=0.6]{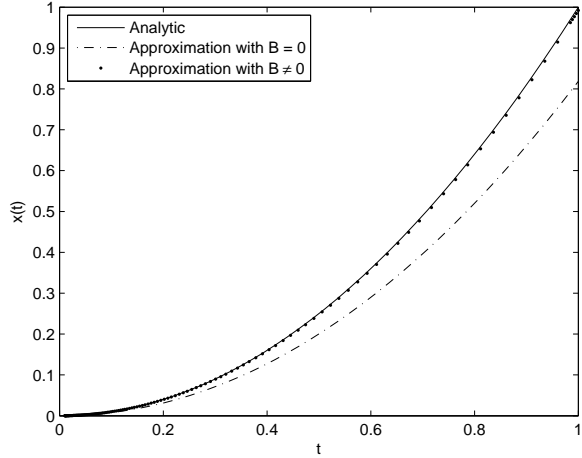}
  \end{center}
  \caption{Comparison of our approach to that of \cite{Atan2}}
  \label{cmpAtanMe}
\end{figure}
\end{remark}


\section{Application to the fractional calculus of variations}
\label{FCOV}

The fractional calculus of variations consists in the study
of dynamic optimization problems in which the objective functional
or constraints depend on derivatives and/or integrals
of non-integer order. This is a recent and promising
research subject, under strong current research (see, \textrm{e.g.},
\cite{Frederico,Malinowska,Mozyrska} and references therein).
Here we show how to use expansions to transform a fractional
problem into a classical one, where we can benefit from the vast number
of techniques available in the field. Consider the following
fractional variational problem including a left Riemann--Liouville
fractional derivative, $\LDa$, of order $\a \in (0,1)$:
\begin{equation}
\label{intJ}
\begin{aligned}
\min\quad &J[x(\cdot)]=\int_a^b L\left(t, x(t),\dot{x}(t),\LDa x(t)\right)dt\\
&x(a)=x_a,\quad x(b)=x_b.
\end{aligned}
\end{equation}
One can deduce fractional necessary optimality equations
to problem \eqref{intJ} of Euler--Lagrange type \cite{ID:207,Tatiana}:
if $x(\cdot)$ is a solution to problem \eqref{intJ}, then
it satisfies the fractional Euler--Lagrange equation
\begin{equation}
\label{myEL}
\frac{\partial L}{\partial x}+\RD \frac{\partial L}{\partial \LDa}
-\frac{d}{dt}\left(\frac{\partial L}{\partial \dot{x}}\right)=0.
\end{equation}

There are a few attempts in the literature
to present analytic solutions to fractional variational problems. Simple problems
have been treated in \cite{AlD1}; some other examples are presented in \cite{AtanHam}.
In this work we use the two expansions discussed in Section~\ref{secexpan}
to reduce a fractional problem to a problem with derivatives of integer order.
In order to illustrate the usefulness of our ideas in the area of the calculus
of variations, we need to consider problems \eqref{intJ} with a known exact solution.
Examples~\ref{Ex51} and \ref{Ex52} below are suitable for our purposes,
since the analytic solutions can be easily obtained.
Knowing the exact solutions, we compare the effectiveness
of different approximation methods.

\begin{example}
\label{Ex51}
Let $\a\in (0,1)$. Consider the following minimization problem:
\begin{equation}
\label{exCOV}
\begin{aligned}
\min\quad &J[x(\cdot)]
=\int_0^1 [\LD x(t)-\dot{x}^2(t)]dt\\
&x(0)=0,\quad x(1)=1.
\end{aligned}
\end{equation}
In this case the Euler--Lagrange equation \eqref{myEL} gives
$$
_tD_1^{\a} 1+2\ddot{x}(t)=0,~\text{or}~ \ddot{x}(t)
=-\frac{1}{2\Gamma(1-\a)}(1-t)^{-\a},
$$
which subject to the given boundary conditions has solution
\begin{equation}
\label{solEx51}
x(t)=-\frac{1}{2\Gamma(3-\a)}(1-t)^{2-\a}
+\left(1-\frac{1}{2\Gamma(3-\a)}\right)t+\frac{1}{2\Gamma(3-\a)}.
\end{equation}
\end{example}

The Lagrangian in Example~\ref{Ex51} is linear
with respect to the fractional derivative. This linearity
makes the fractional Euler--Lagrange equation easy to solve.
In a slightly different situation, \textrm{e.g.} Example~\ref{Ex52},
there is no well-known methods to solve the Euler--Lagrange
equation \eqref{myEL}.

\begin{example}
\label{Ex52}
Given $\a\in (0,1)$, consider now the functional
\begin{equation}
\label{Ex52f}
J[x(\cdot)]=\int_0^1 (\LD x(t)-1)^2dt,
\end{equation}
to be minimized subject to the boundary conditions
$x(0)=0$ and $x(1)=\frac{1}{\Gamma(\a+1)}$.
Since the integrand in \eqref{Ex52f} is non-negative,
the functional attains its minimum when $\LD x(t)=1$,
\textrm{i.e.}, for $x(t)=\frac{t^{\a}}{\Gamma(\a+1)}$.
\end{example}


\subsection{Numerical solutions to Example~\ref{Ex51}}
\label{subsec:nsEx51}

We use two different approaches.

\subsubsection{Expansion to integer orders}

Using approximation \eqref{expanInt} for the fractional derivative
in \eqref{exCOV}, we get the approximated problem
\begin{equation}
\label{expanCOV}
\begin{aligned}
\min\quad &\tilde{J}[x(\cdot)]
=\int_0^1 \left[\sum_{n=0}^N
C(n,\a)t^{n-\a}x^{(n)}(t)-\dot{x}^2(t)\right]dt\\
&x(0)=0,\quad x(1)=1,
\end{aligned}
\end{equation}
which is a classical higher-order problem of the calculus of variations
that depends on derivatives up to order $N$.
The corresponding necessary optimality condition
is a well-known result.

\begin{thm}[\textrm{cf.}, \textrm{e.g.}, \cite{Lebedev}]
Suppose that $x(\cdot)\in C^{2N}[a,b]$ minimizes
$$
\int_a^b L(t,x(t),x^{(1)}(t),x^{(2)}(t),\ldots,x^{(N)}(t))dt
$$
with given boundary conditions
\begin{eqnarray*}
x(a)=a_0, & &x(b)=b_0,\\
x^{(1)}(a)=a_1, & &x^{(1)}(b)=b_1,\\
&\vdots&\\
x^{(N-1)}(a)=a_{N-1}, & & x^{(N-1)}(b)=b_{N-1}.
\end{eqnarray*}
Then $x(\cdot)$ satisfies the Euler--Lagrange equation
\begin{equation}
\label{ELN}
\frac{\partial L}{\partial x}-\frac{d}{dt}
\left(\frac{\partial L}{\partial x^{(1)}}\right)
+\frac{d^2}{dt^2}\left(\frac{\partial L}{\partial x^{(2)}}\right)
-\cdots+(-1)^N\frac{d^N}{dt^N}\left(\frac{\partial L}{\partial x^{(N)}}\right)=0.
\end{equation}
\end{thm}
In general \eqref{ELN} is an ODE of order $2N$, depending
on the order $N$ of the approximation we choose,
and the method leaves $2N-2$ parameters unknown. In our example,
however, the Lagrangian in \eqref{expanCOV} is linear
with respect to all derivatives of order higher than two.
The resulting Euler--Lagrange equation is the second order ODE
\begin{equation*}
\sum_{n=0}^N (-1)^nC(n,\a)\frac{d^n}{dt^n}(t^{n-\a})
-\frac{d}{dt}\left[-2\dot{x}(t)\right]=0
\end{equation*}
that has solution
\begin{multline*}
x(t)=-\frac{1}{2\Gamma(3-\a)}\left[\sum_{n=0}^N
(-1)^n\Gamma(n+1-\a)C(n,\a)\right]t^{2-\a}\\
+\left[1+\frac{1}{2\Gamma(3-\a)}\sum_{n=0}^N
(-1)^n\Gamma(n+1-\a)C(n,\a)\right]t.
\end{multline*}
Figure~\ref{expIntFig} shows the analytic solution together with several
approximations. It reveals that by increasing $N$, approximate solutions
do not converge to the analytic one. The reason is the fact that the solution
\eqref{solEx51} to Example~\ref{Ex51} is not an analytic function.
We conclude that \eqref{expanInt} may not be a good choice
to approximate fractional variational problems. In contrast, as we shall see,
the approximation \eqref{expanMom} introduced in this paper leads to good results.
\begin{figure}
\begin{center}
\includegraphics[scale=.6]{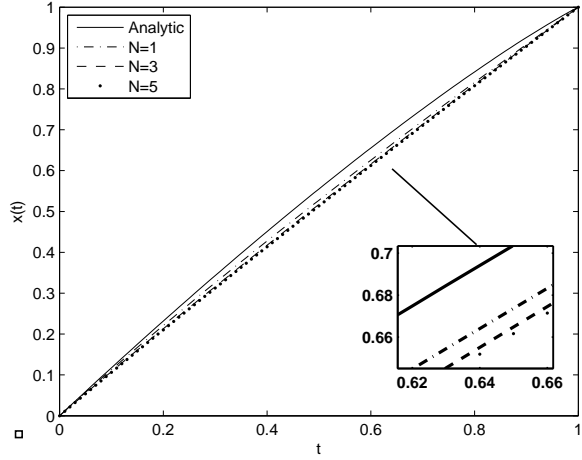}
\caption{Analytic vs. approximate solutions to Example~\ref{Ex51}
using approximation \eqref{expanInt}.}
\label{expIntFig}
\end{center}
\end{figure}


\subsubsection{Expansion through the moments of a function}

If we use \eqref{expanMom} to approximate the optimization problem \eqref{exCOV}, we have
\begin{equation}
\label{momCOV}
\begin{split}
\tilde{J}[x(\cdot)]&=\int_0^1 \left[A(\a,N)t^{-\a}x(t)
+B(\a,N)t^{1-\a}\dot{x}(t)-\sum_{p=2}^N
C(\a,p)t^{1-p-\a}V_p(t)-\dot{x}^2(t)\right]dt,\\
\dot{V}_p(t)&=(1-p)t^{p-2}x(t),\quad p=2,3,\ldots,N,\\
V_p(0)&=0, \quad p=2,3,\ldots,N,\\
x(0)&=0,\quad x(1)=1.
\end{split}
\end{equation}
Problem \eqref{momCOV} is constrained with a set of ordinary differential equations
and is natural to look to it as an optimal control problem \cite{Pontryagin}.
For that we introduce the control variable $u(t) = \dot{x}(t)$. Then,
using the Lagrange multipliers $\lambda_1, \lambda_2,\ldots,\lambda_N$,
and the Hamiltonian system, one can reduce \eqref{momCOV}
to the study of the two point boundary value problem
\begin{equation}
\label{tpbvp}
\left\{
\begin{array}{rl}
\dot{x}(t)&=\frac{1}{2}B(\a,N)t^{1-\a}-\frac{1}{2}\lambda_1(t),\\
\dot{V}_p(t)&=(1-p)t^{p-2}x(t),\quad p=2,3,\ldots,N,\\
\dot{\lambda}_1(t)&=A(\a,N)t^{-\a}-\sum_{p=2}^N(1-p)t^{p-2}\lambda_p(t),\\
\dot{\lambda}_p(t)&=-C(\a,p)t^{(1-p-\a)},\quad p=2,3,\ldots,N,\\
\end{array}\right.
\end{equation}
with boundary conditions
\begin{equation*}
\left\{
\begin{array}{l}
x(0)=0, \\
V_p(0)=0,\quad p=2,3,\ldots,N,
\end{array}
\right.\qquad\left\{
\begin{array}{l}
x(1)=1,\\
\lambda_p(1)=0,\quad p=2,3,\ldots,N,
\end{array}
\right.
\end{equation*}
where $x(0)=0$ and $x(1)=1$ are given. We have $V_p(0)=0$, $p=2,3,\ldots,N$,
due to \eqref{sysVp} and $\lambda_p(1)=0$, $p=2,3,\ldots,N$,
because $V_p$ is free at final time for $p=2,3,\ldots,N$ \cite{Pontryagin}.
In general, the Hamiltonian system is a nonlinear, hard to solve, two point boundary
value problem that needs special numerical methods. In this case, however, \eqref{tpbvp}
is a non-coupled system of ordinary differential equations and is easily solved to give
\begin{equation*}
x(t)=M(\a,N)t^{2-\a}-\sum_{p=2}^N
\frac{C(\a,p)}{2p(2-p-\a)}t^p+\left[ 1-M(\a,N)
+\sum_{p=2}^N \frac{C(\a,p)}{2p(2-p-\a)}\right]t,
\end{equation*}
where
$$
M(\a,N)=\frac{1}{2(2-\a)}\left[ B(\a,N)-\frac{A(\a,N)}{1-\a}
-\sum_{p=2}^N \frac{C(\a,p)(1-p)}{(1-\a)(2-p-\a)} \right].
$$
Figure~\ref{expMomFig} shows the graph of $x(\cdot)$
for different values of $N$.
\begin{figure}
\begin{center}
\includegraphics[scale=.6]{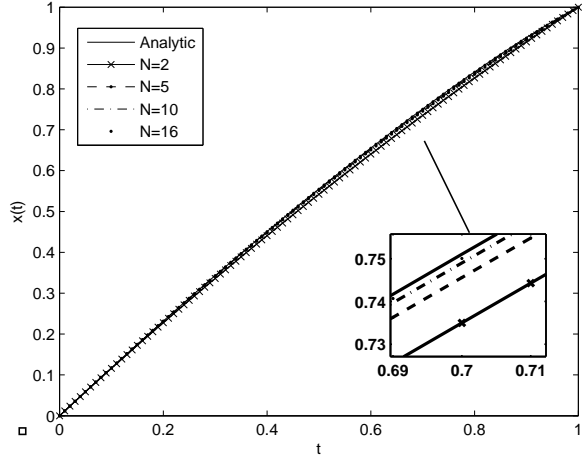}
\caption{Analytic vs. approximate solutions to Example~\ref{Ex51}
using approximation \eqref{expanMom}.}
\label{expMomFig}
\end{center}
\end{figure}


\subsection{Numerical solutions to Example~\ref{Ex52}}
\label{subsec:nsEx52}

Similarly to Example~\ref{Ex51}, it turns out
that expansion \eqref{expanInt} does not provide
a good method while \eqref{expanMom}
leads to good results.


\subsubsection{Expansion to integer orders}

Using \eqref{expanInt} as an approximation
for the fractional derivative in \eqref{Ex52f} gives
\begin{equation*}
\begin{aligned}
\min\quad &\tilde{J}[x(\cdot)]
=\int_0^1 \left(\sum_{n=0}^N C(n,\a)t^{n-\a}x^{(n)}(t)-1\right)^2dt,\\
&x(0)=0,\quad x(1)=\frac{1}{\Gamma(\a+1)}.
\end{aligned}
\end{equation*}
The Euler--Lagrange equation \eqref{ELN}
gives a $2N$ order ODE. For $N \ge 2$ this approach is inappropriate
since the two given boundary conditions $x(0)=0$ and $x(1)=\frac{1}{\Gamma(\a+1)}$
are not enough to determine the $2N$ constants of integration.


\subsubsection{Expansion through the moments of a function}

Let us approximate Example~\ref{Ex52} using \eqref{expanMom}.
The resulting minimization problem has the following form:
\begin{equation}
\label{ex52Mom}
\begin{aligned}
\min\quad &\tilde{J}[x(\cdot)]=\int_0^1 \left[A(\a,N)t^{-\a}x(t)
+B(\a,N)t^{1-\a}\dot{x}(t)-\sum_{p=2}^N
C(\a,p)t^{1-p-\a}V_p(t)-1\right]^2dt,\\
& \dot{V}_p(t)=(1-p)t^{p-2}x(t),\quad p=2,3,\ldots,N,\\
& V_p(0)=0, \quad p=2,3,\ldots,N,\\
&x(0)=0,\quad x(1)=\frac{1}{\Gamma(\a+1)}.
\end{aligned}
\end{equation}
Following the classical optimal control approach
of Pontryagin \cite{Pontryagin} as in Example~\ref{Ex51},
this time with
$$
u(t)=A(\a,N)t^{-\a}x(t)+B(\a,N)t^{1-\a}\dot{x}(t)
-\sum_{p=2}^N C(\a,p)t^{1-p-\a}V_p(t),
$$
we conclude that the solution to \eqref{ex52Mom}
satisfies the system of differential equations
\begin{equation}
\label{ex52tpbvp}
\left\{
\begin{array}{rl}
\dot{x}(t)&=-AB^{-1}t^{-1}x(t)+\sum_{p=2}^NB^{-1}C_pt^{-p}V_p(t)
+\frac{1}{2}B^{-2}t^{2\a-2}\lambda_1(t)+B^{-1}t^{\a-1},\\
\dot{V}_p(t)&=(1-p)t^{p-2}x(t),\quad p=2,3,\ldots,N,\\
\dot{\lambda}_1(t)&=AB^{-1}t^{-1}\lambda_1
-\sum_{p=2}^N(1-p)t^{p-2}\lambda_p(t),\\
\dot{\lambda}_p(t)&=-B^{-1}C(\a,p)t^{-p}\lambda_1,
\quad p=2,3,\ldots,N,\\
\end{array}\right.
\end{equation}
where $A=A(\a,N)$, $B=B(\a,N)$ and $C_p=C(\a,p)$ are defined according
to Section~\ref{secexpan}, subject to the boundary conditions
\begin{equation}
\label{sysB52}
\left\{
\begin{array}{l}
x(0)=0,\\
V_p(0)=0,\quad p=2,3,\ldots,N,
\end{array}
\right.\qquad\left\{
\begin{array}{l}
x(1)=\frac{1}{\Gamma(\a+1)},\\
\lambda_p(1)=0,\quad p=2,3,\ldots,N.
\end{array}
\right.
\end{equation}
The solution to system \eqref{ex52tpbvp}--\eqref{sysB52},
with $N=2$, is shown in Figure~\ref{expMomFig52}.
\begin{figure}
\begin{center}
\includegraphics[scale=.6]{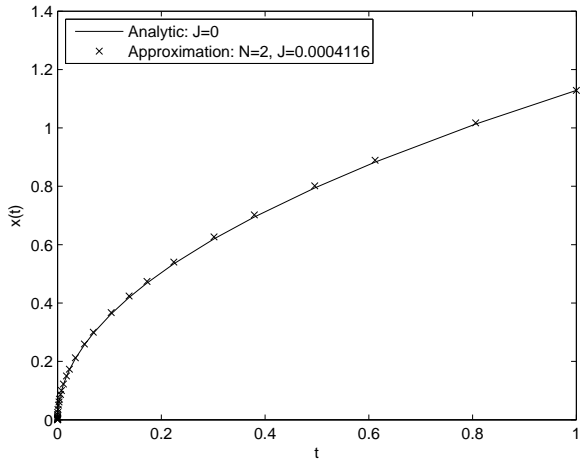}
\caption{Analytic versus approximate solution to Example~\ref{Ex52}
using approximation \eqref{expanMom}.}
\label{expMomFig52}
\end{center}
\end{figure}


\section{Conclusion}

During the last three decades,
several numerical methods have
been developed in the field of fractional calculus.
Some of their advantages, disadvantages,
and improvements, are given in \cite{sug:r3}.
Based on two continuous expansion formulas \eqref{expanIntInf}
and \eqref{expanMomInf} for the left Riemann--Liouville fractional derivative,
we studied two approximations \eqref{expanInt} and \eqref{expanMom} and their applications
in the computation of fractional derivatives. Despite the fact that
the approximation \eqref{expanInt} encounters some difficulties from the presence
of higher-order derivatives, it exhibits better results regarding convergence.
Approximation \eqref{expanMom} can also be generalized to include
higher-order derivatives in the form of \eqref{Gen}. The possibility of using \eqref{expanMom}
to compute fractional derivatives for a set of tabular data was discussed.
Fractional differential equations are also treated successfully. In this case the lack
of initial conditions makes \eqref{expanInt} less useful. In contrast, one can freely increase $N$,
the order of approximation \eqref{expanMom}, and find better approximations.
Comparing with \eqref{expanAtan}, our modification provides better results.
We finished by discussing the solution of fractional variational problems
using the introduced approximations. Similar methods can also be applied
to fractional optimal control problems.

For fractional variational problems, the proposed expansions may be used at two different
stages during the solution procedure. The first approach, the one considered
in Sections~\ref{subsec:nsEx51} and \ref{subsec:nsEx52},
consists in a direct approximation of the problem,
and then treating it as a classical problem,
using standard methods to solve it.
The second approach would be to apply the
fractional Euler--Lagrange equation and then to use the approximations in order
to obtain a classical differential equation. However, the Euler--Lagrange equations
\eqref{myEL} involve right Riemann--Liouville derivatives,
which introduces undesirable difficulties.


\section*{Acknowledgments}

This work was supported by {\it FEDER} funds through
{\it COMPETE} --- Operational Programme Factors of Competitiveness
(``Programa Operacional Factores de Competitividade'')
and by Portuguese funds through the
{\it Center for Research and Development
in Mathematics and Applications} (University of Aveiro)
and the Portuguese Foundation for Science and Technology (FCT),
within project PEst-C/MAT/UI4106/2011
with COMPETE number FCOMP-01-0124-FEDER-022690.
Pooseh was also supported by FCT through
the Ph.D. fellowship SFRH/BD/33761/2009.
The authors are very grateful
to three anonymous referees for valuable remarks and comments,
which significantly contributed to the quality of the paper.


\small



\end{document}